\documentclass[reqno,a4paper]{amsart}

\usepackage[all,numberbysection]{preprint_style}
\usepackage[backend=biber,maxbibnames=5,maxalphanames=5,style=alphabetic,bibencoding=utf8,giveninits,url=false,isbn=false,doi=true,eprint=false]{biblatex}
\renewbibmacro{in:}{}
\AtBeginBibliography{\small}

\usepackage[T1]{fontenc}
\allowdisplaybreaks

\usepackage[top=2.75cm,bottom=2.75cm,left=3cm,right=3cm]{geometry}

\renewcommand{\gap}{\eta^2}
\renewcommand{\gapT}{\eta_T^2}
\newcommand{\gapphi}{\eta_{\phi}^2}
\newcommand{\gappsi}{\eta_{\psi}^2}
\newcommand{\gapphiT}{\eta_{\phi;T}^2}
\newcommand{\gappsiT}{\eta_{\psi;T}^2}
\newcommand{\gapext}{\overline{\eta}^2}

\newcommand{\phiN}{\widehat{\phi}}
\newcommand{\rhoext}{\overline{\rho}_{\mathrm{tot}}^2}
\newcommand{\rhoprimalext}{\overline{\rho}_{u}^2}
\newcommand{\rhodualext}{\overline{\rho}_{\bq}^2}
\renewcommand{\rhoprimal}{\rho_{u}^2}
\renewcommand{\rhodual}{\rho_{\bq}^2}

\begin{document}
	
	%\title[Error estimates via duality for non-conforming fields]{Error estimates via duality for non-conforming fields
	\title[Bregman divergences and estimates via duality]{Bregman divergences and error control via convex duality}
	
	\author[P.A.~Gazca-Orozco]{P.~A. Gazca-Orozco}
	
	\address[P.\ A. Gazca-Orozco]{Faculty of Mathematics and Physics, Charles University, Sokolovská 83, 186 75, Prague, Czech Republic}
	\email{gazca@karlin.mff.cuni.cz}
	
	
	
	\date{\today}
	
	\setcounter{tocdepth}{1} 	
	\begin{abstract}
    Convex duality relations are a useful tool for deriving error estimates for challenging nonlinear and non-smooth variational problems.
    Applied at the continuous level they can deliver nonlinear analogues of the Prager--Synge a posteriori error identity,
    while at the discrete level they allow the derivation of minimal regularity a priori estimates.
 By leveraging elementary properties of Bregman divergences,
 we obtain three results on the error control via convex duality for a general class of problems:
 first, we prove a local efficiency bound for the duality gap error estimator,
 secondly, we derive a guaranteed a posteriori bound for non-conforming fields,
 and finally, we prove a minimal-regularity quasioptimal estimate for a Crouzeix--Raviart discretisation of the $\varphi$-Laplace problem.
	\end{abstract}

    \maketitle

%	\tableofcontents

%  \comment[Lit review]{
%    \begin{itemize}
%      \item \textbf{Stongly monotone problems:}
%        \begin{itemize}
%          \item See Pascal's papers, also using energy methods
%        \end{itemize}
%      \item \textbf{Non-conforming}:
%      \item \textbf{Semilinear}:
%      \item \textbf{Quasilinear problems}:
%        \begin{itemize}
%          \item Earlier $p$-Laplace results: \cite{Ver.1994,CK.2003,CFP.2007}
%          \item \cite{Vee.2002} does convergence using the PDE residual as an estimator.
%          \item \cite{HSW.2008} does $hp$-adaptivity with residual estimators.
%          \item \cite{BDK.2012} proves optimality of AFEM based on residual estimators.
%          \item \cite{ElAEV.2011} does $p$-structure, conforming, not natural distance.
%          \item \cite{EV.2013} does $p$-structure, non-conforming (jumps from the beginning), not natural distance.
%        \end{itemize}
%      \item \textbf{Energy methods:}
%        \begin{itemize}
%          \item \cite{FRV.2024} does strongly non-smooth regularised problems, fully adaptive based on gap estimators.
%            \cite{HMRV.2024} does strongly-monotone Lipschitz with robust estimates wrt nonlinearity.
%        \end{itemize}
%    \end{itemize}
%  }

  \section{Introduction}
  Let $\Omega \subset \RRd$ be a bounded polyhedral Lipschitz domain, and
  consider the problem of minimising a convex energy functional of the form 
  %$I\colon W^{1,p}_0(\Omega) \to \Rext$ of the form 
  \begin{equation}\label{eq:primal_energy}
I(v) \coloneqq \int_\Omega \phi(\cdot,\nabla v) + \int_\Omega \psi(\cdot, v),
  \end{equation}
  where  the functions $\phi\colon \Omega \times \RRd\to \Rext$ and $\psi\colon \Omega \times \RR \to \Rext$ are convex normal integrands (the precise definition will be provided later).
  %such that $\bz\in \RRd \mapsto \phi(x,\bz)\in \Rext$ and $z\in \RR\to\psi(x,z)\in \Rext$ are for a.e.\ $x\in \Omega$ proper, convex, and lower semi-continuous.
  It is well known that under certain coercivity and growth assumptions, there exists a minimiser $u\colon \Omega \to \RR$, hereafter referred to as the \emph{primal solution}, belonging to an appropriate Sobolev $W_0^{1,p}(\Omega)$ (or Orlicz--Sobolev $W_0^{1,\phi}(\Omega)$) space.
  One can associate to this minimisation problem a \emph{dual problem},
  corresponding to the maximisation of the dual energy $D$, defined usually on $L^{p'}(\Omega)$, which here we assume can be written in terms of convex integrands:
  \begin{equation}\label{eq:dual_energy}
D(\br) \coloneqq
- \int_\Omega \phi^*(\cdot,\br) - 
\int_\Omega \psi^*(\cdot, \diver \br),
  \end{equation}
  where $\phi^*$ and $\psi^*$ denote the Fenchel conjugates (with respect to the second variable) of $\phi$ and $\psi$, respectively,
  and $\diver\coloneqq \nabla^*$ represents the adjoint operator of $\nabla$.
  The maximiser of the dual energy $\bq\colon \Omega \to \RRd$ is then known as the \emph{dual solution}.
%  In this work we restrict ourselves to the case where the dual problem is well posed on a smaller subset of $L^{p'}(\Omega)$ such that $\diver\br$ above represents a proper function (for instance $W^{p'}(\diver;\Omega)$), so that the second term in \eqref{eq:dual_energy} is actually the integral of a measurable function.
  In this work we will moreover consider the situation where a primal and a dual solution exist, and furthermore there is no \emph{duality gap}:
  \begin{equation}\label{eq:strong_duality_intro}
    \sup_{\br \in L^{p'}(\Omega)} D(\br )
    =
    \inf_{v\in W^{1,p}(\Omega)} I(v).
  \end{equation}
  This setting covers a wide variety of convex minimisation problems;
  see for instance \cite{ET.1999,Rep08,BK.2024} and the references therein.

  Now, it is a classical observation that the strong duality relation \eqref{eq:strong_duality} implies the following identity for arbitrary (admissible) $v\colon \Omega \to \RR$ and $\br\colon \Omega \to \RR^d$ % $v\in \WonepD{1}$ and $\br\in \WdivpN{1}$ (see e.g.\ \cite{Rep08,BK.2024}): 
  \begin{equation}\label{eq:error_identity_intro}
    \left[I(v)-I(u)\right] + 
    \left[D(\bq)-D(\br)\right]
    = \gap(v,\br),
  \end{equation}
  where $\gap(v,\br)\coloneqq I(v)-D(\br)$ is the so-called duality gap error estimator.
  Since the energy differences on the left-hand-side constitute a natural way of defining the error,
  this constitutes an a posteriori error identity that is reliable and efficient with constant 1 (for earlier results where upper bounds are derived see \cite{RX.1997,Rep.1999,Rep08}).
  This has been exploited in the a posteriori error control of a very wide range of problems;
  see e.g.\ \cite{Rep08,NR.2015,CL.2015,Tra.2024,FRV.2024,BK.2024,DS.2025,GK.2026} and the references therein.

  \subsection{Local efficiency of the duality gap estimator}
With the help of the integration by parts formula (assuming for simplicity that $v$ satisfies zero boundary conditions), the gap estimator can be written as:
\begin{equation}\label{eq:gap_estimator_intro}
  \gap(v,\br)=
  \int_\Omega \left[ \phi(\cdot,\nabla v) - \br\cdot \nabla v + \phi^*(\cdot,\br) \right]
  +
  \int_\Omega \left[ \psi(\cdot, v) - \diver\br \, v + \psi^*(\cdot,\diver\br) \right],
\end{equation}
where thanks to the Fenchel--Young inequality (see \eqref{eq:fenchel_young} below),
 the integrands are pointwise non-negative and vanish if and only if $v=u$ and $\br=\bq$.
 This motivates their use as local error estimators to drive for instance adaptive mesh refinement,
 as done in the works referenced above.
 However,
 a more thorough justification of this requires a \emph{local} efficiency estimate for $\eta(v,\br)$.
 \textbf{This is the first main objective of this work.}
  
 At first glance, it is not obvious how to write down a local estimate, since the error measure on the left-hand-side of \eqref{eq:error_identity_intro} is \emph{global}.
 This apparent difficulty will be overcome with the help of the notion of \emph{Bregman divergences}.
Given a differentiable function $\phi\colon \Omega \times \RRd \to \RR$ (we consider slightly more general functions afterwards),
 one defines its Bregman divergence (with respect to the second variable) as: 
 \begin{equation}
   \mathcal{D}_\phi(\ba,\bb) \coloneqq \phi(\ba)
 - \phi(\bb)- \phi'(\bb)\cdot (\ba-\bb),
  \qquad \text{for }\ba,\bb\in \RRd. 
 \end{equation}
 Even though in general it is not an actual \emph{distance} (since it is not necessarily symmetric),
 it is a natural notion of error associated to the variational problem for \eqref{eq:primal_energy};
 in particular, assuming that $\phi$ is strictly convex, it is non-negative and vanishes if and only if $\ba=\bb$.
 The concept of a Bregman distance has been applied in many contexts,
 ranging from the development of optimisation algorithms to the analysis of PDEs \cite{Kiw.1997,Bur.2016,BK.2023,KS.2024}.
 For us they will prove very useful, in that they will allow us to re-write the generalised Prager--Synge identity \eqref{eq:error_identity_intro} as
 \begin{equation}\label{eq:new_error_identity_intro}
   \int_\Omega \left[\mathcal{D}_\phi(\nabla v,\nabla u)  
+ \mathcal{D}_\psi( v,u)  
+ \mathcal{D}_{\phi^*}(\br,\bq)
+ \mathcal{D}_{\psi^*}(\diver\br,\diver \bq) \right]
= \eta^2(v,\br).
 \end{equation}
 The left-hand-side is simply the collection of error measures (Bregman divergences) associated to the integrands and their conjugates.
 Note that in the quadratic-linear case where $\phi(\cdot,\ba)=\frac{1}{2}|\ba|^2$,
 and $\psi(\cdot,a)=-f(\cdot)a$ with $f\in L^2(\Omega)$,
 the above reduces to the classical Prager--Synge identity for the Laplace equation,
 since the Bregman divergences for $\phi$ and $\phi^*$ are simply the squared Euclidean distance,
 and the ones for $\psi$ and $\psi^*$ vanish ($\mathcal{D}_{\psi^*}(\diver\br,\diver\bq)$ vanishes provided $\diver\br=-f$).
 The key advantage of \eqref{eq:new_error_identity_intro} over \eqref{eq:error_identity_intro} is the presence of integrands which are \emph{locally non-negative},
 which makes it straightforward to pose the question of local efficiency. 
 In this respect,
 assuming that the integrands satisfy what we call a Bregman--Young inequality (Assumption \ref{as:bregman_young}),
 we prove the following local efficiency estimate on an element $T$ (Theorem \ref{thm:local_efficiency}):
 \begin{equation}
   \eta_T^2(v,\br) \leq c_{\phi,\psi}
   \int_T \left[\mathcal{D}_\phi(\nabla v,\nabla u)  
+ \mathcal{D}_\psi(v,u)  
+ \mathcal{D}_{\phi^*}(\br,\bq)
+ \mathcal{D}_{\psi^*}(\diver\br,\diver \bq) \right],
 \end{equation}
 where $\eta^2_T(v,\br)$ is the variant of the estimator \eqref{eq:gap_estimator_intro} localised to the element $T$;
 the constant $c_{\varphi,\psi}$ depends only on the constant from the Bregman--Young inequality.
 For integrands of the form $\phi(\cdot,\ba)=\phiN(|\ba|)$,
 where $\phiN\colon \RRplus\to \RRplus$ is a uniformly convex $N$-function (see Definition \ref{def:unif_convexity}),
 we verify that the Bregman--Young inequality holds, and the constant depends only on the uniform convexity of the integrands;
 in particular it does not depend on non-explicit shape-regularity- or polynomial-degree-dependent quantities.
 As an application this includes for instance the $p$-Laplacian problem for $p\in(1,\infty)$ defined by:
 \begin{equation}
\phi(\cdot,\nabla v) \coloneqq \frac{1}{p}\int_\Omega |\nabla v|^p 
\qquad \psi(\cdot,v) \coloneqq-\int_\Omega f\cdot v,
 \end{equation}
where $f\in L^{p'}(\Omega)$.
In the context of the $p$-Laplacian, one often measures the error using the quantity $\norm{|\nabla v|^{\frac{p-2}{2}}\nabla v - |\nabla u|^{\frac{p-2}{2}}\nabla u }^2_\Omega$ (often called the \emph{natural distance}),
which is the one that leads to optimal error estimates of finite element discretisations, for instance.
This notion of error is actually equivalent to the one based on Bregman divergences that we employ (Proposition \ref{prop:young_n_function}),
so one could write the estimate as:
\begin{equation}\label{eq:intro_plaplace_efficiency}
  \begin{aligned}
  \int_T \left[\tfrac{1}{p}|\nabla v|^2 - \nabla v\cdot \br + \tfrac{1}{p'}|\br|^{p'} \right]
  \leq c_p &\left[
  \norm{|\nabla v|^{\frac{p-2}{2}}\nabla v - |\nabla u|^{\frac{p-2}{2}}\nabla u}^2_{L^2(T)}
  \right.\\
  &\quad+ \left. \norm{|\br|^{\frac{p'-2}{2}}\br - |\bq|^{\frac{p'-2}{2}}\bq}^2_{L^2(T)}
\right],
\end{aligned}
\end{equation}
for arbitrary $v\in W^{1,p}_0(\Omega)$ and $\br \in L^2(\Omega)^d$ with $\diver\br = -f\in L^2(\Omega)$, where $c_p$ depends solely on $p$.
To the best of our knowledge, local estimates for the $p$-Laplacian based on Prager--Synge error relations
have only been written in terms of $W^{1,p}$-norms \cite{ElAEV.2011,EV.2013}, which would deliver a convergence analysis with sub-optimal rates.
In contrast, the estimate \eqref{eq:intro_plaplace_efficiency} employs the natural distance,
thus paving the way to an analysis with optimal behaviour.
Reliability and (global) efficiency of primal-dual gap estimators similar to \eqref{eq:gap_estimator_intro} in the natural norms was derived in \cite{BK.2023,Kal.2024} for piecewise linear approximations, where however some non-explicit constants had to be introduced (for higher-order schemes see also \cite{CT.2021,Tra.2024}).

 \subsection{A posteriori estimate for non-conforming fields}
One of the difficulties of a framework based on energy minimisation is that it is not immediately obvious how to handle non-feasible objects.
For instance, if one wishes to estimate the error related to a discrete solution $v_h$ that arises from a non-conforming method,
it is not allowed to simply plug it into the error relation \eqref{eq:error_identity_intro},
since the energy $I(v)$ is technically set to $+\infty$ in this case.
To deal with this, we will define an extended primal energy on the broken Sobolev space $\Wonepbroken{1}$:
\begin{equation}
\begin{gathered}
%  \overline{I}\colon \Wonepbroken{1} \to \Rext  \\
  \overline{I}(v) \coloneqq 
  \int_\Omega \varphi(\cdot,\discgrad v)
  + \int_\Omega \psi(\cdot,  v),
\end{gathered}
\end{equation}
where $\discgrad\colon \Wonepbroken{1}\to L^{1}(\Omega)^d$ 
 represents a discrete approximation of the gradient operator.
 This could be for instance a DG gradient \cite{BO.09} or simply the broken (piecewise) gradient $\discgrad = \nabla_h$, % $\discdiv=\diver_h$,
but could as well represent an approximation derived from other numerical methods.
We will assume that this approximate gradient coincides with the usual gradient when conforming functions are involved, 
i.e.\ $\discgrad v = \nabla v$ for all $v\in W^{1,1}(\Omega)$.

The energy $\overline{I}$ is defined for a wider range of functions;
note that a function $v$ is admissible as long as pointwise admissibility holds: $\discgrad v \in \dom\varphi$ and $\discdiv v \in \dom\psi$.
At the same time, it will still allow the derivation of useful estimates,
since both energies $I$ and $\overline{I}$ involve \emph{the same} convex integrands $\varphi$ and $\psi$.
We stress that it is not assumed that the energy $\overline{I}$ has a minimiser;
it is merely a tool for deriving error bounds.
In a similar way we define a modified dual energy on the broken space $\Wdivbroken{1}$:
\begin{equation}
\begin{gathered}
%  \overline{D}\colon \Wdivbroken{1} \to \Rextm  \\
  \overline{D}(\br) \coloneqq 
  - \int_\Omega \varphi^*(\cdot,\br)
  - \int_\Omega \psi^*(\cdot, \discdiv \br),
\end{gathered}
\end{equation}
where now $\discdiv\colon \Wdivbroken{1}\to L^{1}(\Omega)$ represents some discrete approximation of the divergence operator, and is such that $\discdiv \br = \diver \br$ for all conforming $\br\in W^{1}(\diver;\Omega)$.

Similarly to the conforming case,
a simple consequence of strong duality $I(u)=D(\bq)$ is the following error identity for possibly non-conforming $v$ and $\br$:
\begin{equation}
  \overline{I}(v)-I(u) 
  + D(\bq) -  \overline{D}(\br)
  = \overline{I}(v) - \overline{D}(\br).
\end{equation}
This relation is however not immediately useful, since $\overline{I}(v)-I(u)$ is for instance not necessarily non-negative and so cannot represent an actual error. However, a few simple computations will yield the extension of the error relation \eqref{eq:new_error_identity_intro}:
\begin{equation}
 \begin{aligned}
   \int_\Omega &\left[\mathcal{D}_\phi(\discgrad v,\nabla u)  
+ \mathcal{D}_\psi(v,u)  
+ \mathcal{D}_{\phi^*}(\br,\bq)
+ \mathcal{D}_{\psi^*}(\discdiv\br,\diver \bq) \right]
= \gapext(v,\br) \label{eq:new_error_identity_intro_ext} \\
 &\qquad\qquad+
(\br-\bq,\discgrad v - \nabla u)_\Omega 
+ 
(\discdiv \br - \diver \bq, v-u)_\Omega,
 \end{aligned}
\end{equation}
 where the extended gap estimator $\gapext$ is simply defined on the broken spaces as 
  \begin{equation}
    \gapext(v,\br) \coloneqq
\int_\Omega \left[ \varphi(\cdot,\discgrad v) 
  - \br\cdot \discgrad v
 + \phi^*(\cdot,\br)
\right]
+
\int_\Omega \left[ \psi(\cdot,v) 
  -v\, \discdiv \br  
 + \psi^*(\cdot,\discdiv\br)
\right].
  \end{equation}
  Now, recalling that $\int_\Omega (\bs \cdot\nabla s + s\diver \bs)=0$ for all conforming $s$ and $\bs$ (with appropriate boundary conditions), we can swap for example the exact solution $u$ for an arbitrary conforming function $s$ in the right-hand-side, assuming that $\br$ is div-conforming.
  An application of a Bregman--Young inequality with $\varepsilon$ to handle the last two terms (Assumption \ref{as:bregman_young_eps}) will then yield the following bound (see Theorem \ref{thm:non_conforming_estimate}):
      \begin{equation}\label{eq:estimate_Hdiv_conforming_intro}
\rhoext(v,\br) \lesssim
\gapext(v,\br) 
+
\inf_{s\in \WonepD{1}}\left[ 
  \int_\Omega \mathcal{D}_{\phi}(\discgrad v,\nabla s)
  +
  \int_\Omega \mathcal{D}_{\psi}(v,s)
\right],
      \end{equation}
      where the constant depends again only on the Bregman--Young inequality constant;
      as before, this will apply to uniformly convex integrands,
      covering in particular the $p$-Laplace problem. 
      The last term in \eqref{eq:estimate_Hdiv_conforming_intro} measures the distance to the conforming space $W^{1,1}(\Omega)$,
      and could then be interpreted as a non-conformity estimator for the function $v$.
      A similar estimate also holds for non $W^{1}(\diver;\Omega)$-conforming $\br$, assuming that $v$ is conforming (see \eqref{eq:estimate_H1_conforming}).
      In practice one would estimate the infimum from above by choosing a particular conforming reconstruction of $v$ (or $\br$); 
      this has been done for problems posed on $H^1(\Omega)$ (e.g.\ for the Laplace equation with $\phi(s)=\frac{1}{2}|s|^2$) in various works, such as \cite{ElAE.2004,EV.2015,AF.2018,C-F.2025},
      where it was shown for instance that a potential reconstruction based on local Poisson solves yields a locally efficient polynomial-robust estimator; see Remark \ref{rem:non_conforming} for more details.
      For problems with $p$-structure we are aware of the work \cite{EV.2013},
      which however included a residual jump estimator from the beginning as a non-conformity measure,
      but as mentioned in the previous section, derived local bounds in the sub-optimal $W^{1,p}$-norms.

      The idea of using an extended energy to handle non-feasible functions was employed recently in \cite{GK.2026} to derive estimates for the incompressible Stokes system without requiring exactly divergence-free velocities,
      and for linear elasticity without requiring exact symmetry of the stress.
      In Section \ref{sec:obstacle} we show another example of a guaranteed a posteriori bound for the obstacle problem, without requiring exact positivity of the multiplier.
      It is conceivable that the method would apply to a general class of convex optimisation problems with constraints,
      where only an additional estimator measuring the distance to the feasible set is needed; this will be the subject of future work.

 \subsection{Quasioptimal a priori estimate for uniformly convex energies}

 The main ingredients in the derivation of the error identity \eqref{eq:new_error_identity_intro} (or \eqref{eq:error_identity_intro}) were the strong duality principle \eqref{eq:strong_duality_intro} and the integration  by parts formula.
 These two ingredients are also available at the discrete level,
 provided appropriate discretisations are employed;
 for instance, if the primal and dual problems are discretised with the Crouzeix--Raviart element $V^h$,
 and the lowest-order Raviart--Thomas element $\Sigma^h$, respectively (for the precise definitions see Section \ref{sec:FE}),
 then an integration by parts formula holds, without requiring any jump terms \cite{BW.2021}:
  \begin{equation}
\int_\Omega  (v_h\diver \br_h + \br_h\cdot \nabla_h v_h)
= \int_{\partial\Omega} \br_h\cdot \bn  v_h
\qquad 
\forall\, \br_h\in \Sigma^h,\, v_h\in V^h,
  \end{equation}
  where $\nabla_h$ represents the piecewise gradient operator.
In addition, if an appropriate discrete versions of the primal and dual energies are considered,
then a discrete strong duality relation will hold.
For simplicity take for example the $p$-Laplace energy on the Crouzeix--Raviart space with zero boundary conditions $V^h_0$ 
(a more general setting will be considered in Section \ref{sec:quasioptimality}); the discrete primal energy then takes the form
\begin{equation}
  I_h(v_h)\coloneqq \frac{1}{p}\int_\Omega |\nabla_h v_h|^p 
  - \int_\Omega f \Pi_hv_h,
\end{equation}
where $p\in (1,\infty)$, $f\in L^{p'}(\Omega)$, and $\Pi_h$ represents the $L^2$-orthogonal projection into piecewise constants. The corresponding dual problem is then:
\begin{equation}
  D_h(\br_h) \coloneqq -\frac{1}{p'}\int_\Omega |\Pi_h\br_h|^{p'} - \chi^{\Omega}_{\{- \Pi_h f\}}(\diver\br_h),
\end{equation}
where the second term is an indicator function that enforces the constraint $\diver\br_h = -\Pi_h f$ a.e.\ in $\Omega$.
If the discrete primal and dual solutions are denoted $u_h\in V^h_0$ and $\bq_h\in \Sigma^h$, respectively,
then discrete strong duality holds \cite{BK.2023}: $I_h(u_h)=D_h(\bq_h)$.
This allows one to mimic the continuous analysis at the discrete level, and derive a discrete analogue of the Prager--Synge identity:
\begin{equation}\label{eq:discrete_prager_synge_intro}
  \int_\Omega \mathcal{D}_\varphi(\nabla_h v_h, \nabla_h u_h)
  +
  \int_\Omega \mathcal{D}_{\varphi^*}(\Pi_h\br_h,\Pi_h\bq_h)
  =
\int_\Omega 
[\varphi(\nabla_h v_h) - 
\Pi_h\br_h\cdot \nabla_h v_h + \varphi^*(\Pi_h\br_h)],
\end{equation}
where $\phi$ is simply defined as $\phi(\bs)= \frac{1}{p}|\bs|^p$,
so $\phi^*(\bs)= \frac{1}{p'}|\bs|^{p'}$.
Such an identity characterises the discrete error (the distance to the discrete solutions) in terms of computable quantities: it is pointwise non-negative and vanishes if and only if $(v_h,\br_h)=(u_h,\bq_h)$;
an analogous identity had been recently derived  in the context of linear elasticity and the incompressible Stokes system \cite{GK.2026} and the Signorini problem \cite{BGK.2025}.
Given the computable nature of the right-hand-side (note also that no unknown constants are present),
it can be used to drive inexact nonlinear solvers; in the context of the $p$-Laplacian the upper bound for conforming schemes was applied to a Ka\v{c}anov iterative scheme in \cite{DS.2025}. Earlier works such as \cite{EV.2013} employed instead an error measure based on $W^{1,p}$-norms, which is suboptimal.

In this work we leverage the identity \eqref{eq:discrete_prager_synge_intro} to derive a \emph{quasioptimal a priori estimate} for the Crouzeix--Raviart and Raviart--Thomas discretised problems.
Similar results were derived recently in \cite{GK.2026} for the Laplace equation, and the linear elasticity and incompressible Stokes systems.
The idea is in essence to employ an analogous estimate to \eqref{eq:intro_plaplace_efficiency} to turn the right-hand-side of \eqref{eq:discrete_prager_synge_intro} into the distance between the exact solution and arbitrary discrete functions,
leading to the equivalence between the discretisation error and the best-approximation error (see Theorem \ref{thm:quasioptimal}):
\begin{equation}\label{eq:a_priori_intro}
%\rhotot(\nabla_h u_h,\Pi_h \bq_h)
  \int_\Omega \mathcal{D}_\varphi(\nabla_h u_h, \nabla u)
  +
  \int_\Omega \mathcal{D}_{\varphi^*}(\Pi_h\bq_h,\bq)
\sim
\inf_{v_h \in V^h_0}
\int_\Omega \mathcal{D}_{\phi}(\nabla_h v_h, \nabla u)
  +
  \inf_{\substack{\br_h \in \Sigma^h \\ \diver \br_h = - \Pi_h f}}
  \int_\Omega   \mathcal{D}_{\phi^*}(\Pi_h\br_h, \bq).
\end{equation}
This estimate does not contain indeterminate constants, and they depend solely on the uniform convexity of $\phi$ (see \eqref{eq:uniform_convexity}).
Also, unlike the textbook analysis of Crouzeix--Raviart discretisations of elliptic problems,
this is a minimal-regularity result, requiring only $u\in W^{1,p}_0(\Omega)$ and $\bq \in W^{p'}(\diver;\Omega)$,
and thus guarantees convergence to minimal-regularity solutions;
convergence rates can then be obtained if additional regularity is at hand, by employing appropriate interpolation opearators; for the pioneering works with such suboptimal assumptions see \cite{LY.2001,CL.2015}.
Other a priori error bounds for the $p$-Laplacian with improved regularity assumptions include \cite{Kal.2024} and
\cite{Sto.2025}, 
where the bounds were based on \emph{medius analysis},
which have the downside that they contain unknown shape-regularity-dependent constants and must include data oscillation terms.
Truly minimal-regularity estimates with no oscillation terms like \eqref{eq:a_priori_intro} were derived in
\cite{BGKR.2026}, which however include again non-explicit constants, require a jump stabilisation term, and the implementation of a so-called smoothing operator on the right-hand-side (see also related results in \cite{DHKZ.2025}).

  \section{Preliminaries}

  \subsection{Convexity and Bregman distances}
  We first recall that the \emph{effective domain} of a convex function $\phi\colon \RRn \to \Rext$ is defined as $\dom\phi\coloneqq \{\ba \in \RRn \mid \phi(\ba)<+\infty\}$, and we say that $\phi$ is proper if $\dom\phi \not=\emptyset$.
The subdifferential of a proper convex function $\varphi\colon \RR^n\to \Rext$ at a point $\ba\in \dom\varphi$ is defined as  
\begin{equation}\label{eq:subdifferential}
\partial\varphi(\ba)\coloneqq 
\{\br\in \RR^n \mid \varphi(\ba) + \br\cdot (\bb-\ba) \leq \varphi(\bb)\text{ for all }\bb\in \RRn \},
\end{equation}
(setting $\partial\phi(\ba)=\emptyset$ if $\ba\not\in \dom\phi$); its domain is simply $\dom\partial\phi=\{\ba\in\RRn \mid \partial\phi(\ba)\not= \emptyset\}$.
The (Fenchel) conjugate of a proper convex function $\phi\colon \RR\to \Rext$ is the 
proper convex and lower semicontinuous function $\phi^*\colon \RRn\to \Rext$ defined by:
\begin{equation}
\phi^*(\ba^*)\coloneqq 
\sup_{\ba\in \RRn}
\left[ 
\ba^*\cdot \ba - \phi(\ba)
\right].
\end{equation}
Moreover, for any proper convex lower semicontinuous function $\phi\colon \RRn\to\Rext$,
the \emph{Fenchel--Young} inequality holds:
\begin{equation}\label{eq:fenchel_young}
\ba^*\cdot \ba \leq 
\phi(\ba) + \phi^*(\ba^*)
\qquad \forall\,\ba,\ba^*\in\RRn,
\end{equation}
with equality being identical to the duality relations:
\begin{equation}\label{eq:equality_young}
\ba^*\cdot \ba =
\phi(\ba) + \phi^*(\ba^*)
\quad \Leftrightarrow \quad 
\ba \in \partial\phi^*(\ba^*) 
\quad \Leftrightarrow \quad 
\ba^* \in \partial\phi(\ba).
\end{equation}

The (generalised) Bregman distance \cite{Kiw.1997} $\mathcal{D}^{\br}_{\varphi}\colon \dom\varphi\times \dom\partial\varphi$ with respect to the proper convex function $\varphi$ and a given subgradient $\br\in\partial\varphi(\bb)$ is defined as
\begin{equation}\label{eq:bregman}
  \mathcal{D}^{\br}_\phi(\ba,\bb)\coloneqq 
    \phi(\ba) - \phi(\bb)
    -  \br\cdot ( \ba - \bb).
  \end{equation}
We remark that we are collecting these helpful concepts for functions on $\RRn$, since they will be applied at the pointwise level to the integrands in \eqref{eq:primal_energy} and \eqref{eq:dual_energy}.
  However, we note that they can be applied at the infinite-dimensional level as well to obtain global estimates \cite{Bar.2015,Bur.2016}.
  Whenever $\varphi$ is differentiable at a point $\ba \in\RRn$, one has $\partial\varphi(\ba)=\{\varphi'(\ba)\}$ and we ommit the upper index: $\mathcal{D}_{\varphi}\coloneqq\mathcal{D}_{\varphi}^{\varphi'(\bb)}$.
%  \begin{equation}
%  \mathcal{D}_\phi(\ba,\bb)\coloneqq 
%    \phi(\ba) - \phi(\bb)
%    -  \varphi'(\bb)\cdot ( \ba - \bb).
%  \end{equation}
  The following lemma summarises basic properties of Bregman divergences that will be useful in our analysis.
  \begin{lemma}[\cite{CT.1993,Kiw.1997,Bur.2016}]\label{lem:bregman}
Suppose $\phi\colon \RRn \to \Rext$ is convex and proper. Then the following properties hold:
\begin{enumerate}
  \item {[Non-negativity]} One has $\mathcal{D}^{\br}_\phi(\ba,\bb)\geq 0$ for all $\ba\in \dom\phi$, $\bb\in \dom\partial\phi$, and $\br\in\partial\phi(\bb)$.
    Moreover,
    $\mathcal{D}^{\br}_{\phi}(\ba,\ba) = 0$ for all $\ba\in \dom\phi \cap \dom\partial\phi$.

  \item {[Positivity]} If $\phi$ is strictly convex, then $\mathcal{D}^{\br}_\phi(\ba,\bb)>0$ if $\ba\not=\bb$.

%  \item {[Shift]} Let $r \in \partial\phi(w)$. The function $v\mapsto \mathcal{D}^{r}_\phi(v,w)$ is convex, and its subdifferential is given by $\partial_v\mathcal{D}_\phi^{r}(v,w)=\partial\phi(v) - r$.

%    \textcolor{morado}{[Also, might be useful, its conjugate is $H^*(p)= \phi^*(p+r)-\phi^*(r)$.]}

  \item \label{lem:bregman_sharing_subgradients} {[Sharing of subgradients]} For $\br \in \partial\phi(\bb)$ one has:
    \begin{equation}
      \mathcal{D}^{\br}_\phi(\ba,\bb) =
      \phi(\ba) + \phi^*(\br) - \br\cdot \ba
    \end{equation}
    In particular $\mathcal{D}^{\br}_\phi(\ba,\bb)=0$ if and only if $\br\in \partial\phi(\ba)$.

  \item {[Conjugate distance]} Let $\br\in \partial \phi(\bb)$ and $\bs \in \partial\phi(\ba)$. Then 
    \begin{equation}
      \mathcal{D}^{\br}_\phi(\ba,\bb)
      =
      \mathcal{D}^{\ba}_{\phi^*}(\br,\bs).
    \end{equation}
\end{enumerate}
%{
%  \color{morado}
%  [Another, possibly useful ones?]
%  $$ \mathcal{D}_\phi^r(v,w) + \mathcal{D}^s(w,v) = \langle r-s,w-v\rangle$$
%}
  \end{lemma}
  While $\mathcal{D}^{\br}_\phi$ is not necessarily symmetric, and thus is not a distance in the traditional sense,
  it is a nonlinear notion of distance tailored to the structure of the problem and has been used in the analysis of various nonlinear problems;
  see e.g.\ \cite{Kiw.1997,Bur.2016,BK.2023,KS.2024}.

  There is one further property that will be crucial in the subsequent analysis,
  which we now introduce as an assumption. This will be verified for instance in the particular case in which $\varphi$ is a uniformly convex $N$-function (see Proposition \ref{prop:young_n_function} below).

  \begin{assumption}[Bregman--Young inequality]\label{as:bregman_young}
    Let $\ba_1\in \dom\phi$, $\ba_2\in\dom\partial\phi$, $\br_1\in \dom\phi^*$, and $\br_2\in \dom\partial\phi^*$,
    as well as the subgradients $\bp\in\partial\phi(\ba_2)$, $\bb \in \partial\phi^*(\br_2)$ be arbitrary.
    Suppose further that  the duality relation $\ba_2 \in \partial \phi^*(\br_2)$ holds.
We say that $\varphi$ satisfies the Bregman--Young inequality if one has
\begin{equation}\label{eq:bregman_young}
|(\ba_1 - \ba_2)\cdot(\br_1-\br_2) |
\lesssim 
\mathcal{D}^{\bp}_{\phi}(\ba_1,\ba_2) + 
\mathcal{D}^{\bb}_{\phi^*}(\br_1,\br_2).
\end{equation}
We will mostly apply this with $\bp = \br_2$ and $\bb=\ba_2$.
%\commentalexei{Check if this implies the one with $\varepsilon$. 
%Check the subgradients?}
%Also, I may need to write these using the optimality relations and not in general...}
  \end{assumption}

  \subsection{Functional setting}

  Throughout this work, the domain $\Omega\subset \RRd$ ($d\in\RRd$) will represent a bounded Lipschitz domain.
  We will make use of standard notation for Lebesgue $(L^p(\Omega))$ and Sobolev spaces $W^{1,p}(\Omega)$ ($p\in [1,\infty]$).
  The conjugate exponent $p'\in[1,\infty]$ of a number $p\in [1,\infty]$ is characterised through the relation $\frac{1}{p}+\frac{1}{p'}=1$.
  The inner product and norm of the space $L^2(\Omega)$ will be denoted, respectively,
  by $(\cdot,\cdot)_\Omega$ and $\norm{\cdot}_\Omega$.
  We will also make use of the spaces
  \begin{equation}
    W^p(\diver;\Omega) \coloneqq \{\br \in L^p(\Omega)^d \mid \diver \br \in L^p(\Omega)\},
    \quad p\in [1,\infty].
  \end{equation}
Regarding boundary conditions,
let us split the boundary $\partial \Omega$ into a Dirichlet $\Gamma_D$ and Neumann $\Gamma_N$ boundaries,
such that $\Gamma_D$ and $\Gamma_N$ are (relatively) open, disjoint, and such that $\partial\Omega\setminus(\Gamma_D\cup\Gamma_N)$ has zero $(d-1)$-dimensional measure.
Homogeneous mixed boundary conditions will often be considered for simplicity,
so the following spaces will be useful:
\begin{align}
  \WonepD{p} &\coloneqq 
  \{v\in W^{1,p}(\Omega) \mid v|_{\Gamma_D}=0\text{ in }W^{\frac{1}{p'},p}(\Gamma_D)\}, \\
  \WdivpN{p'} &\coloneqq 
  \{\br \in  W^{p'}(\diver;\Omega) \mid (\br,\nabla v)_\Omega + (\diver\br,v)_\Omega = 0 \quad \forall\, v\in \WonepD{p}\}. \label{eq:WdivN}
\end{align}

  The main goal in this work is to characterise the (Bregman) distance to the minimiser of the convex energy functional
  \begin{equation}\label{eq:primal_energy2}
I(v) \coloneqq \int_\Omega \phi(\cdot,\nabla v) + \int_\Omega \psi(\cdot, v),
  \end{equation}
  We will assume here that $\phi\colon \Omega \times \RRd \to \Rext$ and $\psi\colon \Omega \times \RR \to \Rext$ are \emph{normal convex integrands} \cite{Roc.1971},
  meaning that they are convex functions in the second variable, such that their epigraphical mapping (defined similarly for $\psi$)
  $x\in \Omega \to \mathrm{epi}\phi(x,\cdot)\subset \RRd \times \RR$ is closed-valued and measurable.
  This includes for instance
  \begin{itemize}
    \item Carathéodory integrands: $x\mapsto \phi(x,\ba)$ is measurable for all $\ba\in \RRd$ and $\ba\mapsto \phi(x,\ba)$ is continuous for almost all $x\in\Omega$.
    \item Autonomous lower-semicontinuous integrands: $\phi(x,\ba)=\widetilde{\phi}(\ba)$ with $\widetilde{\phi}$ lower semicontinuous.
    \item Indicator functionals of a measurable, closed- and convex-valued mapping $x\in \Omega \mapsto K(x) \subset \RRd$: a function  $\chi_K \colon \Omega \times \RRd \to \Rext$ of the type:
      \begin{equation}\label{eq:indicator}
\chi_K(x,\ba) \coloneqq 
\chi_{K(x)}(\ba) 
= 
\left\{ 
  \begin{array}{cc}
    0 & \text{if }\ba\in K(x),\\
    +\infty & \text{if }\ba \not\in K(x),\\
  \end{array}
  \right.
      \end{equation}
  \end{itemize}
  We note also that the convex conjugate (defined with respect to the second variable) $\varphi^*\colon \Omega \times \RRd \to \Rext$ is in this case also a convex normal integrand.
  The integrals in \eqref{eq:primal_energy2} are then interpreted in an extended sense;
  see \cite{RF.1998} for more details.
  In particular, we can interpret for instance for a function of the type $\phi(x,\cdot)= \widetilde\phi(x,\cdot)+\chi_{K(x)}$,
  where $\widetilde{\phi}$ is a Carathéodory integrand, the integral as:
  \begin{equation}
\int_\Omega \phi(\cdot,\nabla v) =
\left\{
  \begin{array}{cc}
    \int_\Omega \widetilde{\phi}(\cdot,\nabla v) & \text{if }\nabla v(\cdot)\in K(\cdot) \text{ a.e.\ in }\Omega \\ 
    +\infty & \text{otherwise.}
  \end{array}
\right.
  \end{equation}
  From Fenchel duality theory \cite{ET.1999}, it is known that one can associate to the minimisation problem above a dual maximisation problem, associated to the dual energy
  \begin{equation}\label{eq:dual_energy2}
D(\br) \coloneqq
- \int_\Omega \phi^*(\cdot,\br) - 
\int_\Omega \psi^*(\cdot, \diver \br).
  \end{equation}
  A finer analysis of function spaces relating to the existence of the primal and dual solutions does not play a huge role in this work,
  and we will simply include existence as an assumption.
  This is classically a consequence of the direct method of the calculus of variations,
  under appropriate growth and coercivity assumptions. %\todo{CITE}

%  We will simply assume that appropriate coercivity and growth assumptions hold,
%  that guarantee the existence of a primal solution (minimiser) $u\in W^{1,1}_0(\Omega)$ with $I(u) = \min_{v\in \WonepD{1}}I(v)<+\infty$.
  \begin{assumption}[Existence and strong duality]\label{as:existence}
    There exists a function $u\in \WonepD{1}$ (called the primal solution) such that 
    \begin{equation}
      I(u)= \min_{v\in \WonepD{1}}I(v)<+\infty.
    \end{equation}
    There is a function $\bq \in \WdivpN{1}$ (called the dual solution) such that
    \begin{equation}
      D(\bq)= \max_{\br\in \WdivpN{1}}D(\br)>-\infty.
    \end{equation}
  Moreover, strong duality holds (or the absence of a duality gap): 
  \begin{equation}\label{eq:strong_duality}
I(u)=D(\bq).
  \end{equation}
  \end{assumption}
  We remark that often the primal and dual solutions will often belong to smaller spaces;
  for instance, if $|\ba|^p\lesssim \phi(\cdot,\ba)|\lesssim 1+|\ba|^p$ a.e.\ in $\Omega$ with $p\in (1,\infty)$,
  one will typically have $u\in \WonepD{p}$ and $\bq\in \WdivpN{p'}$.

    As a consequence of the Fenchel--Young inequality \eqref{eq:fenchel_young}, a weak duality relation holds:
    \begin{equation}
I(v)\geq D(\br) 
\qquad \forall\, v\in \WonepD{1},\, 
\br\in \WdivpN{1}.
    \end{equation}
    (Here the left- and right-hand side are allowed to take the value $+\infty$ and $-\infty$, respectively.)
    The strong duality relation \eqref{eq:strong_duality} is hence a stronger statement.
    In this case, the following duality relations will be satisfied \cite[Prop.~3.1]{BK.2024}:
    \begin{subequations}\label{eq:optimality_relations}
    \begin{align}
      \bq\cdot \nabla u &= 
\phi(\cdot,\nabla u) 
+ \phi^*(\cdot,\bq)
\qquad \;\;\,\text{a.e. in }\Omega, \label{eq:optimality_relations_phi}\\
      \diver\bq\cdot u &= 
\psi(\cdot, u) 
+ \psi^*(\cdot,\diver\bq)
\qquad \text{a.e. in }\Omega.\label{eq:optimality_relations_psi}
    \end{align}
  \end{subequations}

  \subsection{Uniformly convex $N$-functions}

  An application of particular interest in this work will be the case where $\phi(\cdot,\ba)=\widehat{\phi}(|\ba|)$,
  with $\phiN\colon \RRplus\to\RRplus$ being a uniformly convex $N$-function;
  we follow here the treatment from \cite{DFTW.2020}.

  \begin{definition}[(Uniformly convex) $N$-functions]\label{def:unif_convexity}
A function $\phiN\colon \RRplus \to\RRplus$ is said to be an $N$-function if
\begin{enumerate}
  \item $\phiN$ is continuous and convex;
  \item There is a right-continuous and non-decreasing function $\phiN'\colon \RRplus \to \RRplus$ such that $\phiN(t)=\int_0^t\phiN'(s)\dd s$ ,
    and satisfying also $\phiN'(0)=0$, $\phiN'(s)>0$ for $s>0$, and $\lim_{s\to\infty}\phiN'(s)=\infty$.
\end{enumerate}
The $N$-function $\phiN$ is then said to be \emph{uniformly convex} if there are constants $c_{\mathrm{uc}},C_{\mathrm{uc}}>0$ such that 
\begin{equation}\label{eq:uniform_convexity}
  c_{\mathrm{uc}} \frac{\phiN'(s)-\phiN'(t)}{s-t}
  \leq 
  \frac{\phiN'(s)}{s}
  \leq 
  C_{\mathrm{uc}} \frac{\phiN'(s)-\phiN'(t)}{s-t}
\end{equation}
Given a uniformly convex $N$-function $\phiN$, one defines its conjugate as $\phiN^*(t)\coloneqq \max_{s\geq 0}(st - \phiN(s))$.
The function $\phiN^*$ is in turn also a uniformly convex $N$-function.
  \end{definition}

  A classical example of a uniformly convex $N$-function is $\phiN(s)=\frac{1}{p}(\delta^2 + s^2)^{\frac{p}{2}}-\frac{1}{p}\delta^p$ with $p\in (1,\infty)$ and $\delta\geq 0$,
leading for instance to the usual $p$-Laplace operator when $\delta=0$.
More generally, if we define the nonlinear function $\mathcal{A}\colon \RRd \to \RRd$, and its inverse $\mathcal{A}^{-1}\colon \RRd\to\RRd$ as
\begin{equation}
  \mathcal{A}(\ba) \coloneqq 
  \left\{ 
    \begin{array}{cc}
      \frac{\phiN'(|\ba|)}{|\ba|}\ba & \text{if }\ba\not=0, \\
      0 & \text{if }\ba=0,
    \end{array}
    \right.
    \qquad 
    \mathcal{A}^{-1}(\bb) \coloneqq 
  \left\{ 
    \begin{array}{cc}
      \frac{{\widehat{\phi}^{*}}{'}(|\bb|)}{|\bb|}\bb & \text{if }\bb\not=0, \\
      0 & \text{if }\bb=0,
    \end{array}
    \right.
\end{equation}
then the optimality relation \eqref{eq:optimality_relations_phi} can be recast as 
\begin{equation}\label{eq:optimality_relations_pLaplace}
  \bq = \mathcal{A}(\nabla u) 
  \quad \Longleftrightarrow\quad 
  \nabla u = \mathcal{A}^{-1}(\bq)
  \quad \text{a.e.\ in }\Omega.
\end{equation}
In this setting the (primal) solution $u$ will belong to the Orlicz--Sobolev space $W_D^{1,\phi}(\Omega)$,
and assuming for instance that $\psi(\cdot,a)=f(\cdot)a$, with $f\in (L^{\phi}(\Omega))^*$,
the optimality condition (Euler--Lagrange equation) associated to the primal problem is the so-called $\phi$-Laplace equation:
\begin{equation}\label{eq:phi_laplace}
  \int_\Omega \mathcal{A}(\nabla u)\cdot \nabla v 
  = \int_\Omega fv 
  \qquad \forall\, v\in W_D^{1,\phi}(\Omega).
\end{equation}
In particular we note that in this case Assumption \ref{as:existence} is satisfied (see e.g.\ \cite{DK.2008,DS.2025}).
Evidently, our setting includes also semilinear problems where $\psi(\cdot,a)=\widehat{\psi}(|a|)$, with uniformly convex $\widehat{\psi}$ (in which case $u\in W_D^{1,\phi}(\Omega)\cap L^\psi(\Omega)$).

Now, it is well known that in the analysis of problems of the type \eqref{eq:phi_laplace},
when measuring errors it is crucial to employ a nonlinear notion of distance specifically tailored to the structure of the problem;
this is what is usually known as a \emph{quasi-norm} or \emph{natural distance} \cite{BL.1994,EL.2005,RD.2007,DR.2007,DK.2008}.
A useful tool when working with these notions of distance is that of a \emph{shifted $N$-function} $\phiN_a$ (with a shift $a\geq 0$) associated to an $N$-function $\phiN$.
Following \cite{DFTW.2020} (cf.\ \cite{DE.2008}), the shifted $N$-function is defined in terms of its derivative: 
\begin{equation}\label{eq:shifted_N_function}
\phiN_a(t)\coloneqq \int_0^t \phiN'_a(s)\dd s,
\qquad \text{where} \qquad
\phiN_a'(s)\coloneqq \frac{\phiN'(\max\{a,s\})}{\max\{a,s\}}s 
\qquad s\geq 0.
\end{equation}
Defining 
\begin{equation}
  \mathcal{V}(\ba) \coloneqq 
  \left\{ 
    \begin{array}{cc}
      \sqrt{\frac{\phiN'(|\ba|)}{|\ba|}}\ba & \text{if }\ba\not=0, \\
      0 & \text{if }\ba=0,
    \end{array}
  \right.
\end{equation}
then one can choose as a natural distance between two vectors $\ba,\bb\in \RRd$
any the following equivalent quantities
(the constants here and in the estimates \eqref{eq:triangle_unif_convex}--\eqref{eq:shift_change} below depend only on the uniform convexity constants $c_{\mathrm{uc}},C_{\mathrm{uc}}$ from \eqref{eq:uniform_convexity}):
\begin{equation}\label{eq:natural_distance}
  \phiN_{|\ba|}(|\bb - \ba|)
  \sim 
  (\phiN_{|\ba|})^*(|\mathcal{A}(\bb)-\mathcal{A}(\ba)|)
  \sim
  |\mathcal{V}(\ba)-\mathcal{V}(\bb)|^2
  \sim 
  (\mathcal{A}(\ba)-\mathcal{A}(\bb))\cdot (\ba-\bb).
\end{equation}
Note that in the linear case (i.e.\ when $\phiN(t)=\frac{1}{2}t^2$), these reduce to the usual (squared) Euclidean distance $|\ba-\bb|^2$.

As further useful properties, we note that for the shifted function it holds $(\phiN_a)^*=(\phiN^*)_{\phiN'(a)}$,
and one has that $\phiN_a$ is uniformly convex, whenever $\phiN$ is, with uniform convexity constants independent of the shift $a$.
Moreover, one has the following triangle-type inequality: 
\begin{equation}\label{eq:triangle_unif_convex}
\phiN(|s+t|) \lesssim \phiN(|s|) + \phiN(|t|)
\qquad \forall\, s,t\geq 0.
\end{equation}
There is also a useful variation of the Fenchel--Young inequality:
for $\varepsilon>0$ there is $C_\varepsilon>0$, %(which depends on the uniform convexity constants $c_{uc},C_{uc}$),
such that 
\begin{equation}\label{eq:fenchel_young_eps}
s t \leq \varepsilon \phiN(s) + C_\varepsilon \phiN^*(t)
\quad\forall\, s,t \geq 0.
\end{equation}
%with a constant that depends on $c_{uc},C_{uc}$.
Also very handy are the following shift-change estimates: for any $\varepsilon>0$ there is $C_\varepsilon>0$, such
that for any $\ba,\bb\in\RRd$ and $t\geq 0$ it holds 
\begin{subequations}\label{eq:shift_change}
  \begin{align}
    \phiN_{|\ba|}(t) 
    &\leq (1+C_\varepsilon)\phiN_{|\bb|}(t) 
    + \varepsilon |\mathcal{V}(\ba) - \mathcal{V}(\bb)|^2 \\
%    + \varepsilon \phiN_{|\ba|}(|\bb-\ba|), \\
    (\phiN_{|\ba|})^*(t) 
    &\leq (1+C_\varepsilon)(\phiN_{|\bb|})^*(t) 
%    + \varepsilon \phiN_{|\ba|}(|\bb-\ba|),
    + \varepsilon |\mathcal{V}(\ba) - \mathcal{V}(\bb)|^2.
  \end{align}
\end{subequations}
These estimates also hold with the roles of $\varepsilon$ and $C_\varepsilon$ reversed.

In this work we will employ the (pointwise) Bregman divergence as a main error measure.
The following proposition shows that this is equivalent to the natural quantities in \eqref{eq:natural_distance}, thus guaranteeing  the validity of the Bregman--Young inequality \eqref{eq:bregman_young}.

  \begin{proposition}\label{prop:young_n_function}
    Suppose the integrand $\phi\colon \RRd\to \RRd$ is such that $\phi(\ba)=\phiN(|\ba|)$ for all $\ba\in\RRd$,
    for some uniformly convex $N$-function $\phiN$.
Then the Bregman divergence 
\begin{equation}
  \mathcal{D}_{\phi}(\ba,\bb) 
  \coloneqq 
  \phi(\ba) - \phi(\bb) - \mathcal{A}(\bb)\cdot (\ba-\bb)
\qquad \text{for }\ba,\bb\in\RRd,
\end{equation}
is equivalent to the natural distance from \eqref{eq:natural_distance}, with constants depending only on the uniform convexity constants $c_{\mathrm{uc}},C_{\mathrm{uc}}$ from \eqref{eq:uniform_convexity}.
In particular, $\varphi$ satisfies Assumption \ref{as:bregman_young}.
  \end{proposition}
  \begin{proof}
    The stated equivalence was essentially already proved in \cite[Lem.~42]{DFTW.2020},
    where the focus was placed on the case where $\mathcal{A}(\bb)=0$,
    in which case the energy difference $\phi(\ba)-\phi(\bb)$ is equivalent to the natural distance.
   However, the same proof yields the claimed equivalence above.

   To verify Assumption \ref{as:bregman_young}, take arbitrary $\ba_1,\ba_2,\br_1,\br_2\in \RRd$ such that $\br_2=\mathcal{A}(\ba_2)$.
   Combining the Fenchel--Young inequality \eqref{eq:fenchel_young} for the shifted $N$-function $\phiN_{|\ba_2|}$ with the fact that $(\phiN_{|\ba_2|})^* = \phiN^*_{\phiN'(|\ba_2|)}=\phiN^*_{|\br_2|}$, we obtain:
\begin{align*}
 | (\ba_1-\ba_2)\cdot (\br_1 - \br_2)| 
  &\leq 
 (\phiN_{|\ba_2|})^*(|\br_1 - \br_2|)
  +
 \phi_{|\ba_2|}(|\ba_1 - \ba_2|)
  \\
  & =
 \phi^*_{|\br_2|}(|\br_1 - \br_2|)
  +
 \phi_{|\ba_2|}(|\ba_1 - \ba_2|)
  \\ &\lesssim 
  \mathcal{D}_{\phi}(\ba_1, \ba_2)
  +
  \mathcal{D}_{\phi^*}(\br_1, \br_2),
\end{align*}
where in the last line we simply applied the equivalence above.
  \end{proof}

  \subsection{Finite element approximation}\label{sec:FE}
We will consider in this paper a family of conforming shape-regular triangulations $\{\mathcal{T}_h\}_{h>0}$ of $\Omega$; we will assume that the boundary facets from are fully contained in either $\Gamma_D$ or $\Gamma_N$.
The set of facets in the triangulation will be denoted by $\facets$, which can be decomposed  into sets if interior $\facetsint$ and boundary facets $\facetsbdry$;
the set of boundary facets can in turn be decomposed $\facetsbdry=\facetsD\cup\facetsN$ into Dirichlet $\facetsD$ and Neumann $\facetsN$ facets.
We will denote by $\mathbb{P}_k(T)$ the space of polynomials of degree at most $k\in\mathbb{N}$ on a given element $T\in\triang$.
Define also the \emph{broken Sobolev spaces} for $p\in[1,\infty]$:
\begin{align}
  \Wonepbroken{p} &\coloneqq 
  \{v_h \in L^p(\Omega) \mid v_h|_T\in W^{1,p}(T) \text{ for all }T\in\triang\},\\
  \Wdivbroken{p} &\coloneqq
  \{\br_h \in L^p(\Omega)^d \mid \br_h|_T\in W^{p}(\diver;T) \text{ for all }T\in\triang\}.
\end{align}
For each $v_h\in \Wonepbroken{p}$ we denote by $\nabla_h v_h \in L^p(\Omega)^d$ the \emph{broken gradient},
which is characterised by $(\nabla_h v_h)|_T = \nabla(v_h|_T)$, for all $T\in\triang$.
The space of \emph{broken polynomials of degree at most $k$} is then defined as:
\begin{equation}
  \mathbb{P}_k(\triang) \coloneqq 
  \{v_h \in L^\infty(\Omega)\mid v_h|_T\in \mathbb{P}_k(T)\text{ for all }T\in\triang\}.
\end{equation}
The local $L^2$-orthogonal projection onto piecewise constant fields $\Pi_h\colon L^1(\Omega)\to\mathbb{P}_0(\triang)$ is defined by the relation 
\begin{equation}
  \Pi_h v_h|_T \coloneqq \frac{1}{|T|}\int_T v_h \qquad 
  \forall\, T\in \triang.
\end{equation}
Denoting the midpoint of a facet $F\in \facets$ by $x_F$, the \emph{Crouzeix--Raviart space} can be defined as:
\begin{equation*}
V^h \coloneqq 
\{v_h \in \mathbb{P}_1(\triang)\mid v_h \text{ is continuous at }x_F \text{ for all }F\in\facets \}.
\end{equation*}
Similarly, if $x_T$ denotes the barycenter of an element $T\in\triang$, the lowest-order \emph{Raviart--Thomas} space can be defined as:
\begin{equation*}
\Sigma^h\coloneqq 
\{\br_h\in W^1(\diver;\Omega)\mid \br_h|_T(x) = a_T + b_T(x-x_T),
\text{ with }a_T\in\RRd,\,b_T\in\RR,\text{ for all }T\in\triang\}.
\end{equation*}
In turn, the spaces with homogeneous boundary conditions imposed are denoted by 
\begin{align}
  \CRDirichlet &\coloneqq 
\{v_h\in V^h \mid v_h(x_F)=0\text{ for all }F\in \facetsD\},\\
\RTNeumann &\coloneqq 
\{\br_h\in \Sigma^h \mid \br_h|_F\cdot \bn_F = 0 \text{ for all }F\in\facetsN\}.
\end{align}
These spaces satisfy remarkable compatibility conditions;
for instance, the following useful integration by parts formula holds \cite{BW.2021}:
\begin{equation}\label{eq:int_by_parts_discrete}
  (\diver \br_h , v_h)_\Omega 
  + (\br_h,\nabla_h v_h)_\Omega = 0
  \qquad \forall\, \br_h\in \RTNeumann,\, v_h\in\CRDirichlet.
\end{equation}

  \section{Local efficiency of the gap estimator}\label{sec:local_efficiency}
For a given normal integrand $\phi\colon \Omega \times \RRd\to \Rext$,
we define its associated convexity error estimator as 
\begin{equation}
\gapphi(\bp,\br)
\coloneqq 
  \int_\Omega \left[ \phi(\cdot,\bp) - \br\cdot \bp + \phi^*(\cdot,\br) \right],
\end{equation}
where $\bp,\bq\in L^1(\Omega)^d$; note that the estimator is allowed to take the value $+\infty$ if e.g.\ $\phi(\cdot,\bp)$ is not integrable,
but thanks to the Fenchel--Young inequality, it is always non-negative.
In fact, by the same token 
the integrand is pointwise non-negative and so the estimator can be written in terms of local non-negative contributions:
  \begin{gather*}
    \gapphi(\bp,\br) = \sum_{T\in \triang}  \gapphiT(\bp,\br) \\ 
    \gapphiT(\bp,\br) \coloneqq 
  \int_T \left[ \phi(\cdot,\bp) - \br\cdot \bp + \phi^*(\cdot,\br) \right]
  \end{gather*}
  With this at hand, we can now introduce the duality gap estimator as 
  \begin{equation}\label{eq:gap_estimator}
    \begin{aligned}
      \gap(v,\br) &\coloneqq 
      \gapphi(\nabla v,\br) + \gappsi(v,\diver\br) \\
                  &=
  \int_\Omega \left[ \phi(\cdot,\nabla v) - \br\cdot \nabla v + \phi^*(\cdot,\br) \right]
  +
  \int_\Omega \left[ \psi(\cdot, v) - \diver\br \, v + \psi^*(\cdot,\diver\br) \right],
    \end{aligned}
  \end{equation}
and its corresponding local version:
  \begin{equation}\label{eq:gap_estimator_local}
 %   \begin{aligned}
      \gapT(v,\br) \coloneqq 
      \gapphiT(\nabla v,\br) + \gappsiT(v,\diver\br),
      %\\                  &=
%  \int_\Omega \left[ \phi(\cdot,\nabla v) - \br\cdot \nabla v + \phi^*(\cdot,\br) \right]
%  +
%  \int_\Omega \left[ \psi(\cdot, v) - \diver\br \, v + \psi^*(\cdot,\diver\br) \right],
 %   \end{aligned}
  \end{equation}
  where $v\in \WonepD{1}$ and $\br\in \WdivpN{1}$.
  Note that by \eqref{eq:equality_young}, the estimator vanishes if and only if the optimality conditions \eqref{eq:optimality_relations} hold, meaning that $v$ and $\br$ are in fact the solution of the problem.

%  Now, it is a classical observation that the strong duality relation \eqref{eq:strong_duality} implies the following identity for arbitrary $v\in \WonepD{1}$ and $\br\in \WdivpN{1}$ (see e.g.\ \cite{BK.2024}): 
%  \begin{equation}
%    \left[I(v)-I(u)\right] + 
%    \left[D(\bq)-D(\br)\right]
%    = \gap(v,\br).
%  \end{equation}
%  Since the energy differences on the left-hand-side constitute a natural way of defining the error,
%  this constitutes an a posteriori error identity that is reliable and efficient with constant 1.

  As mentioned in the introduction, strong duality yields the a posteriori error identity \eqref{eq:error_identity_intro},
  which, while useful and elegant, is based on a global error measure. 
  In order to localise we will make use of error measures based on Bregman divergences:
  \begin{subequations}\label{eq:convexity_measures}
    \begin{gather}
\rhotot(v,\br) \coloneqq 
\rhoprimal(v)
+ \rhodual(\br)
\\ 
\rhoprimal(v) \coloneqq 
\int_\Omega \mathcal{D}_\varphi(\nabla v, \nabla u) 
+ \int_\Omega \mathcal{D}_\psi(v,u)
\\
\rhodual(\br) \coloneqq 
\int_\Omega \mathcal{D}_{\varphi^*}(\br,\bq)
+ \int_\Omega \mathcal{D}_{\psi^*}(\diver \br, \diver \bq)
\end{gather}
  \end{subequations}
  Here for simplicity we ommit the superindices associated to the subgradients: $\mathcal{D}^{\bq}_{\varphi}(\nabla v,\nabla u) =\mathcal{D}_\phi(\nabla v,\nabla u)$, $\mathcal{D}^{\diver \bq}_{\psi}(v,u)=\mathcal{D}_\psi(v,u)$,
  $\mathcal{D}^{\nabla u}_{\phi^*}(\br,\bq)=\mathcal{D}_{\phi^*}(\br,\bq)$,
  $\mathcal{D}^{u}_{\psi^*}(\diver\br,\diver\bq)=\mathcal{D}_{\psi^*}(\diver\br,\diver\bq)$.
  These are natural error measures that now involve pointwise non-negative integrands,
  thus making it possible to write down localised estimates.
  We first verify that indeed the error identity \eqref{eq:error_identity_intro} can be re-written in terms of these error measures. 
  
  \begin{proposition}[Generalised Prager--Synge identity]\label{prop:identity_bregman}
    Suppose that Assumption \ref{as:existence} is satisfied.
    Then the following error identity holds for all $v\in \WonepD{1}$ and $\WdivpN{1}$:
    \begin{equation}
\rhotot(v,\br)
= \gap(v,\br).
    \end{equation}
  \end{proposition}

  \begin{proof}
We start with the left-hand-side and apply the strong duality relation \eqref{eq:strong_duality}:
\begin{align*}
\rhotot(v,\br)
&= \int_\Omega \left[\phi(\nabla v) - \phi(\nabla u) - \bq\cdot(\nabla v - \nabla u) \right]
+ \int_\Omega \left[\psi(v) - \psi(u) - \diver\bq\cdot(v -u) \right]
\\ &\quad + 
 \int_\Omega \left[\phi^*(\br) - \phi^*(\bq) - \nabla u \cdot(\br - \bq) \right]
+ \int_\Omega \left[\psi^*(\diver \br) - \psi^*(\diver \bq) - u\cdot(\diver\br -\diver\bq) \right]
 \\&= 
\gap(v,\br),
\end{align*}
where in the last line we used integration by parts, noticing that no boundary terms appear thanks to the homogeneous boundary conditions.
Without loss of generality we assumed that the integrals in question are finite, so that these computations are meaningful;
for instance, for $\varphi$ with $p$-growth for $p\in(1,\infty)$, we exploit $\int_\Omega \br\cdot \nabla v = - \int_\Omega \diver \br v $ for all $\br\in W^{p'}_N(\diver;\Omega)$, $v\in W^{1,p}_D(\Omega)$, and so forth.
\end{proof}

%  Consider the duality gap estimator \textcolor{morado}{[defined for?]}
%  \begin{equation}\label{eq:gap_estimator}
%  \gap(v,\br) \coloneqq 
%  \int_\Omega \left[ \phi(\cdot,\nabla v) - \br\cdot \nabla v + \phi^*(\cdot,\br) \right]
%  +
%  \int_\Omega \left[ \psi(\cdot, v) - \diver\br \, v + \psi^*(\cdot,\diver\br) \right]
%  \end{equation}
%  This estimator can be written in terms of local non-negative contributions:
%  \begin{gather*}
%    \gap(v,\br) = \sum_{T\in \triang}  \gapT(v,\br) \\ 
%    \gapT(v,\br) \coloneqq 
%  \int_T \left[ \phi(\cdot,\nabla v) - \br\cdot \nabla v + \phi^*(\cdot,\br) \right]
%  +
%  \int_T \left[ \psi(\cdot, v) - \diver\br \, v + \psi^*(\cdot,\diver\br) \right]
%  \end{gather*}
We now proceed to prove the local efficiency of $\gapT$ with respect to a localised version of the error measure $\rhotot$.

\begin{theorem}[Local efficiency of the gap estimator]\label{thm:local_efficiency}
  Suppose that Assumption \ref{as:existence} holds, and that the integrands $\phi$ and $\psi$ satisfy Assumption \ref{as:bregman_young}.
  Then the duality gap estimator is locally efficient; more precisely, one has for all $v\in \WonepD{1}$ and $\br\in \WdivpN{1}$:
  \begin{equation}\label{eq:local_efficiency}
  \gapT(v,\br) \lesssim \int_T \left[ \mathcal{D}_\phi(\nabla v,\nabla u)
+ 
\mathcal{D}_{\psi}(v,u)
+
\mathcal{D}_{\phi^*}(\br,\bq)
+
\mathcal{D}_{\psi^*}(\diver\br,\diver\bq)\right],
\end{equation}
where the constant depends only on the constant of the Bregman--Young inequality \eqref{eq:bregman_young}.
\end{theorem}

\begin{proof}
%Denote the first and second term in $\eta^2_{\mathrm{gap};T}(v,\br)$ by $\eta^2_{\phi;T}(v,\br)$
%and $\eta^2_{\psi;T}(v,\br)$, respectively.
We only carry out the estimate corresponding to $\phi$; the one for $\psi$ is analogous.
From Lemma \ref{lem:bregman}\ref{lem:bregman_sharing_subgradients} and the optimality condition $\nabla u \in \partial \phi^*(\cdot,\bq)$ it follows that
\begin{align*}
  \eta^2_{\phi;T}(\nabla v, \br)
  &=
  \eta^2_{\phi;T}(\nabla v, \br)
  \pm \int_T \phi^*(\cdot, \bq) 
  \mp \int_T \nabla v \cdot \bq
\\&
=
\int_T \left[
  \mathcal{D}_{\phi}(\nabla v, \nabla u)
  - \nabla v\cdot(\br-\bq)
  + \phi^*(\cdot,\br)
  -\phi^*(\cdot,\bq)
\right]
\\&
=
\int_K \left[
  \mathcal{D}_{\phi}(\nabla v, \nabla u)
  +
  \mathcal{D}_{\phi^*}(\br, \bq)
  - (\nabla v-\nabla u)\cdot(\br-\bq)
\right]
\end{align*}
An application of the Bregman--Young inequality \eqref{eq:bregman_young} proves the claim.
\end{proof}

\begin{remark}
  From the proof of \ref{thm:local_efficiency}, it is evident that in fact one always has 
%  \commentalexei{is it possible to check this with an $\varepsilon$?}
  \begin{align}
0 \leq \phi(\cdot, \ba) - \bs\cdot \ba + \phi^*(\cdot, \bs)
=
-(\ba - \bb)\cdot (\bs - \bt)
+ \mathcal{D}_\phi(\ba,\bb)
+ \mathcal{D}_{\phi^*}(\bs,\bt)
  \end{align}
  whenever $\bb \in \partial\phi^*(\cdot,\bt)$,
  which means that Assumption \ref{as:bregman_young} is needed only to ensure the other direction:
  \begin{equation}
-(\ba-\bb)\cdot(\bs-\bt)
\lesssim
 \mathcal{D}_\phi(\ba,\bb)
+ \mathcal{D}_{\phi^*}(\bs,\bt)
  \end{equation}
\end{remark}

%\begin{corollary}
%  \todo{Prove the same with strong monotonicity and modified error measure?}
%\end{corollary}

\begin{remark}
In practice often $v$ is taken as a numerical approximation of the primal problem $u_h$,
and $\br$ is constructed through some flux equilibration procedure; see e.g.\ \cite{BS.2008,EV.2013,EV.2015} and references therein.
For these reconstruction procedures, it is possible to prove for particular cases a counterpart of the local bound \eqref{eq:local_efficiency} that only involves primal errors, (see for instance also \cite{EGSV.2022}).
We expect that a similar analysis can be carried out for the nonlinear systems considered in this work, by relying on specific flux equilibration procedures.
\end{remark}

%\newpage
  \section{A posteriori bound for non-conforming functions}
  One difficulty with a framework based on energy functionals is that the energy is in principle defined for functions belonging to a space like $\Wonepzero{p}$ (for integrands with $p$-growth for instance),
  and is usually set to $+\infty$ for non-conforming functions.
  How to derive estimates for non-conforming functions is therefore not immediately obvious.
However,
since the energies under consideration here are defined in terms of  integrands,
it is possible to consider modified energies that are defined on non-conforming spaces and then derive error bounds by working in a pointwise-manner.

To this end, we define the extended primal energy, defined on the broken Sobolev space $\Wonepbroken{1}$:
\begin{equation}
\begin{gathered}
  \overline{I}\colon \Wonepbroken{1} \to \Rext  \\
  \overline{I}(v) \coloneqq 
  \int_\Omega \varphi(\cdot,\discgrad v)
  + \int_\Omega \psi(\cdot,  v),
\end{gathered}
\end{equation}
where $\discgrad\colon \Wonepbroken{1}\to L^{1}(\Omega)^d$ 
 represents a discrete approximation of the gradient operator.
 This could be for instance a DG gradient \cite{BO.09} or simply the broken (piecewise) gradient $\discgrad = \nabla_h$, % $\discdiv=\diver_h$,
but could as well represent an approximation derived from other numerical methods.
We will assume that this approximate gradient coincides with the usual gradient when conforming functions are involved, 
i.e.\ $\discgrad v = \nabla v$ for all $v\in W^{1,1}(\Omega)$.

The energy $\overline{I}$ is defined for a wider range of functions;
note that a function $v$ is admissible as long as pointwise admissibility holds: $\discgrad v \in \dom\varphi$ and $\discdiv v \in \dom\psi$.
At the same time, it will still allow the derivation of useful estimates,
since both energies $I$ and $\overline{I}$ involve \emph{the same} convex integrands $\varphi$ and $\psi$.
We stress that it is not assumed that the energy $\overline{I}$ has a minimiser;
it is merely a tool for deriving error bounds.

In a similar way we define a modified dual energy on the broken space $\Wdivbroken{1}$:
\begin{equation}
\begin{gathered}
  \overline{D}\colon \Wdivbroken{1} \to \Rextm  \\
  \overline{D}(\br) \coloneqq 
  - \int_\Omega \varphi^*(\cdot,\br)
  - \int_\Omega \psi^*(\cdot, \discdiv \br),
\end{gathered}
\end{equation}
where now $\discdiv\colon \Wdivbroken{1}\to L^{1}(\Omega)$ represents some discrete approximation of the divergence operator.
Similarly to the gradient operator, we assume that $\discdiv \br = \diver \br$ for all conforming $\br\in W^{1}(\diver;\Omega)$.

Similarly to the conforming case,
a simple consequence of strong duality $I(u)=D(\bq)$ is the following identity:
\begin{equation}
  \overline{I}(v)-I(u) 
  + D(\bq) -  \overline{D}(\br)
  = \overline{I}(v) - \overline{D}(\br),
\end{equation}
which holds for any $v\in \Wonepbroken{1}$ and $\br\in \Wdivbroken{1}$ such that $\discgrad v\in \dom \varphi$, $v\in \dom\psi$, $\br\in \dom \varphi^*$ and $\discdiv \br \in \dom\psi^*$ a.e.\ in $\Omega$.
While it seems at first glance an error identity with a similar structure to the conforming one,
one issue is that for instance $\overline{I}(v)-I(u)$ is not necessarily non-negative (since $u$ is not necessarily a minimiser of the modified energy) and therefore cannot represent an error.
However, after a reformulation in terms of Bregman divergences, 
this relation will lead to a meaningful a posteriori estimate.
The following lemma will shed light on this relationship;
for this let us define the extended gap estimator $\gapext \colon \Wonepbroken{1}\times \Wdivbroken{1}\to \Rext$ as
  \begin{equation}\label{eq:gap_estimator_ext}
    \gapext(v,\br) \coloneqq
\int_\Omega \left[ \varphi(\cdot,\discgrad v) 
  - \br\cdot \discgrad v
 + \phi^*(\cdot,\br)
\right]
+
\int_\Omega \left[ \psi(\cdot,v) 
  -v\, \discdiv \br  
 + \psi^*(\cdot,\discdiv\br)
\right],
  \end{equation}
  and the generalised notions of  error (similarly to \eqref{eq:convexity_measures}, we ommit the superindices corresponding to the subgradients):
  \begin{subequations}\label{eq:extended_convexity_measures}
    \begin{gather}
\rhoext(v,\br) \coloneqq 
\rhoprimalext(v)
+ \rhodualext(\br)
\\ 
\rhoprimalext(v) \coloneqq 
\int_\Omega \mathcal{D}_\varphi(\discgrad v, \nabla u) 
+ \int_\Omega \mathcal{D}_\psi(v,u)
\\
\rhodualext(\br) \coloneqq 
\int_\Omega \mathcal{D}_{\varphi^*}(\br,\bq)
+ \int_\Omega \mathcal{D}_{\psi^*}(\discdiv \br, \diver \bq)
\end{gather}
  \end{subequations}

\begin{lemma}
Then one has:
\begin{equation}
  \overline{I}(v) - \overline{D}(\br)
  = \gapext(v,\br)
     + (\discdiv \br, v)_\Omega + (\br,\discgrad v)_\Omega
\end{equation}
Consequently, the following error identity holds:
\begin{equation}\label{eq:error_identity_extended}
\rhoext(v,\br) 
= 
\gapext(v,\br)
+
(\br-\bq,\discgrad v - \nabla u)_\Omega 
+ 
(\discdiv \br - \diver \bq, v-u)_\Omega.
\end{equation}
\end{lemma}
\begin{proof}
  Taking admissible functions $v\in \Wonepbroken{1}$ and $\br\in \Wdivbroken{1}$ with respect to the modified energies,
  we compute the `duality gap':
\begin{align*}
  \overline{I}(v) - \overline{D}(\br)
  &=
\int_\Omega \left[ \varphi(\cdot,\discgrad v) 
 + \phi^*(\cdot,\br)
\right]
+
\int_\Omega \left[ \psi(\cdot,v) 
 + \psi^*(\cdot,\discdiv\br)
\right]
\\  &=
\int_\Omega \left[ \varphi(\cdot,\discgrad v) 
  - \br\cdot \discgrad v
 + \phi^*(\cdot,\br)
\right]
+
\int_\Omega \left[ \psi(\cdot,v) 
  -v\, \discdiv \br  
 + \psi^*(\cdot,\discdiv\br)
\right]
\\ 
    &\quad + (\discdiv \br, v)_\Omega + (\br,\discgrad v)_\Omega
 \\ &=
 \overline{\eta}(v,\br)
     + (\discdiv \br, v)_\Omega + (\br,\discgrad v)_\Omega
\end{align*}
Adding and substracting the necessary terms to obtain the corresponding Bregman divergences then yields the claimed identity \eqref{eq:error_identity_extended}.
\end{proof}

In order to obtain a useful a posteriori estimate it is then necessary to handle the last two terms in the identity \eqref{eq:error_identity_extended}.
The elementary fact that 
\begin{equation}
  (\br, \nabla s)_\Omega + (\diver \br , s)_\Omega = 0
\end{equation}
holds for all conforming functions $s\in \WonepD{1}$ and $\br\in \WdivpN{1}$ will be useful in what follows;
of course, we assume that they have the appropriate integrability to render the integrals finite (otherwise the statement is vacuously true).
In particular, note that if $\br$ is $W^{1}(\diver;\Omega)$-conforming,
one can substitute the exact solution $u$ for any arbitrary conforming function $s\in \WonepD{1}$:
\begin{equation}\label{eq:error_identity_Hdivconforming}
\rhoext(v,\br) 
= 
\gapext(v,\br)
+
(\br-\bq,\discgrad v - \nabla s)_\Omega 
+ 
(\diver \br - \diver \bq, v-s)_\Omega.
\end{equation}
Similarly, if $v$ is $\WonepD{1}$-conforming, one can swap the exact flux $\bq$ for an arbitrary conforming flux $\bs\in \WdivpN{1}$:
\begin{equation}\label{eq:error_identity_H1conforming}
\rhoext(v,\br) 
= 
\gapext(v,\br)
+
(\br-\bs,\nabla v - \nabla u)_\Omega 
+ 
(\discdiv \br - \diver \bs, v-u)_\Omega.
\end{equation}
Thus, to finish the argument we require Young-like inequalities similar to \eqref{eq:bregman_young} but which include a parameter $\varepsilon>0$ that allows one to absorb the distance to the exact solution on the left-hand-side.
This would then yield an error estimate that includes a distance to the conforming subspace,
which can be interpreted as a non-conformity error.

  \begin{assumption}[Bregman--Young inequality with $\varepsilon$]\label{as:bregman_young_eps}
    Let $\ba_1\in \dom\phi$, $\ba_2\in\dom\partial\phi$, $\br_1\in \dom\phi^*$, and $\br_2,\br_3\in \dom\partial\phi^*$,
    be arbitrary.
%    as well as the subgradients $\bp\in\partial\phi(\ba_2)$, $\bb \in \partial\phi^*(\br_2)$ be arbitrary.
    Suppose further that  the duality relation $\ba_2 \in \partial \phi^*(\br_2)$ holds.
We say that $\varphi$ satisfies the Bregman--Young inequality with $\varepsilon$ if for any $\varepsilon>0$ one can estimate:
\begin{equation}\label{eq:bregman_young_eps}
|(\ba_1 - \ba_2)\cdot(\br_1-\br_3)|
\lesssim 
\varepsilon (\mathcal{D}_{\phi}(\ba_1,\ba_2) + 
\mathcal{D}_{\phi^*}(\br_1,\br_2))
+ \mathcal{D}_{\varphi^*}(\br_1,\br_3).
\end{equation}
  \end{assumption}
  The objects $\ba_2$ and $\br_2$ in \eqref{eq:bregman_young_eps} will usually represent the exact solution,
  so the term with $\varepsilon$ can in principle be absorbed on the left-hand-side.
  We note that in the quadratic case ($\phi(s)=\frac{1}{2}|s|^2$) the second term is not really needed,
  but for more general nonlinearities general it might be. The next proposition shows that this is satisfied by uniformly convex integrands.

  \begin{proposition}\label{prop:young_n_function_eps}
    Suppose the integrand $\phi\colon \RRd\to \RRd$ is such that $\phi(\ba)=\phiN(|\ba|)$ for all $\ba\in\RRd$,
    for some uniformly convex $N$-function $\phiN$.
    Then $\phi$ satisfies Assumption \ref{as:bregman_young_eps};
    the constant in the estimate depends only on the uniform convexity constants $c_{\mathrm{uc}},C_{\mathrm{uc}}$ from \eqref{eq:uniform_convexity}.
  \end{proposition}
  \begin{proof}
    An application of the Fenchel--Young inequality with $\varepsilon>0$ \eqref{eq:fenchel_young_eps} with the $N$-function $\phiN^*_{|\br_1|}$ yields:
\begin{align*}
|(\ba_1-\ba_2)\cdot(\br_1-\br_3)|
&\leq \varepsilon \varphi_{|{{\phiN}^*}{'}(\br_1)|} (|\ba_1-\ba_2|) 
+ C_\varepsilon  \varphi^*_{|\br_1|}(|\br_1 - \br_3|)
%\\ 
%&= 
% \varepsilon \int_\Omega \varphi^*_{|\phi'(\discgrad v)|} (|\br-\bq|) 
%+ c_\varepsilon \int_\Omega \varphi_{|\discgrad v|}(\discgrad v - \nabla s)
\end{align*}
where we used the fact that $(\phiN^*_{|\br_1|})^*= ({\widehat{\phi}^{**}})_{|{{\phiN}^*}{'}(\br_1)|} = {\widehat{\phi}}_{|{{\phiN}^*}{'}(\br_1)|}$.
The last term is already of the desired form, 
and to handle the first term we apply the change of shift estimate \eqref{eq:shift_change}
with new shift $|{\phiN^*}{'}(\br_2)|=|\ba_2|$:
\begin{align*}
\varepsilon \varphi_{|{{\phiN}^*}{'}(\br_1)|} (|\ba_1-\ba_2|) 
\lesssim 
\varepsilon \left[ 
\varphi_{|\ba_2|} (|\ba_1-\ba_2|) 
+
\varphi_{|{{\phiN}^*}{'}(\br_2)|} (|{{\phiN}^*}{'}(\br_2)- {{\phiN}^*}{'}(\br_1)|) 
\right]
\end{align*}
The claim then follows from the equivalences \eqref{eq:natural_distance} and Proposition \ref{prop:young_n_function}.
  \end{proof}

As an immediate consequence of the Bregman--Young inequality with $\varepsilon$ \eqref{eq:bregman_young_eps} and the error identities \eqref{eq:error_identity_Hdivconforming} and \eqref{eq:error_identity_H1conforming} we then obtain an a posteriori guaranteed bound.

\begin{theorem}[Guaranteed estimate for non-admissible fields]\label{thm:non_conforming_estimate}
  Suppose that Assumption \ref{as:existence} holds,
  and that the integrands $\phi$ and $\psi$ satisfy Assumption \ref{as:bregman_young_eps}. 
  Then:
  \begin{enumerate}
    \item %{($\WdivpN{1}$-conforming case)}
      If $\br$ is $\WdivpN{1}$-conforming,
      then the following estimate holds: 
      \begin{equation}\label{eq:estimate_Hdiv_conforming}
\rhoext(v,\br) \lesssim
\gapext(v,\br) 
+
\inf_{s\in \WonepD{1}}\left[ 
  \int_\Omega \mathcal{D}_{\phi}(\discgrad v,\nabla s)
  +
  \int_\Omega \mathcal{D}_{\psi}(v,s)
\right]
      \end{equation}
    \item %{($\WdivpN{1}$-conforming case)}
      If $v$ is $\WonepD{1}$-conforming,
      then the following estimate holds: 
      \begin{equation}\label{eq:estimate_H1_conforming}
\rhoext(v,\br) \lesssim
\gapext(v,\br) 
+
\inf_{\bs\in \WdivpN{1}}\left[ 
  \int_\Omega \mathcal{D}_{\phi^*}(\br,\bs)
  +
  \int_\Omega \mathcal{D}_{\psi^*}(\br,\bs)
\right]
      \end{equation}
  \end{enumerate}
  The constant in the estimates depend only on the constant from the Bregman--Young inequality with $\varepsilon$ \eqref{eq:bregman_young_eps}.
\end{theorem}

\begin{remark}\label{rem:non_conforming}
  The estimates \eqref{eq:estimate_Hdiv_conforming} and \eqref{eq:estimate_H1_conforming} can be used in practice as guaranteed upper bounds for the numerical error,
  by setting $s$ in \eqref{eq:estimate_Hdiv_conforming} (resp.\ $\bs$ in \eqref{eq:estimate_H1_conforming}) to be appropriate quasi-interpolants of the non-conforming numerical solution.
  Similar results have been derived the quadratic case $\phi(s)=\frac{1}{2}|s|^2$ \cite{ElAE.2004,EV.2015,AF.2018,C-F.2025},
  where it is known (for instance by using stability properties of nodal averaging quasi-interpolants \cite[Lem.~3.2]{BE.2007}) that one can also swap this for a residual jump estimator:
  \begin{equation}\label{eq:jumps}
    \inf_{s\in \WonepD{2}}\int_\Omega  |\nabla_h v - \nabla s|^2 
    \lesssim \sum_{F \in \facetsint\cup \facetsD} \frac{1}{h_F}\int_F |\jump{v}_F|^2,
    %\norm{h_\Gamma^{-1/2} \jump{v_h} }
  \end{equation}
  where $\jump{v}_F$ denotes the jump of $v$ across a mesh facet $F\in \mathcal{F}_h$,
  which is usually straightforward to implement.
  By using appropriate extensions to the nonlinear setting (e.g.\ \cite[Prop.~A.1]{BK.2023} or \cite[Prop.~3]{BGKR.2026}),
  it should be possible to generalise this to the uniformly convex setting,
  but this is outside of the scope of this work.
A disadvantage of this approach is that the constant in the upper bound ceases to be explicit, a dependence of polynomial degree is introduced,
and it is not clear that the residual jump estimator is locally efficient (unless of course it is itself included in the error measure).
  For other conforming approximations (in the quadratic case) where one still retains guaranteed bounds and local efficiency,
  see for instance \cite{EV.2015,C-F.2025}, where local solves are employed.
  Similar remarks about $\WdivpN{1}$-conformity apply to \eqref{eq:estimate_H1_conforming}; see e.g.\ \cite{EV.2015,EGSV.2022}.

%  \todo{For problems with a $p$-structure in a non-conforming setting, we are only aware of \cite{EV.2013}, where the jump residual estimator was included from the beginning as a non-conformity estimator;
%  however, the $W^{1,p}_D(\Omega)$-norm was employed instead of the natural distance, which in general leads to suboptimal behaviour.}
\end{remark}

\subsection{Data oscillation}

Another situation where it is sometimes difficult to produce feasible functions $v$ and $\br$ is when a balance equation needs to be satisfied.
As an example let us go back to the $\phi$-Laplace problem with the primal energy $I\colon W^{1,\phi}_0(\Omega)\to \RR$ given by
\begin{equation}\label{eq:primal_phi_laplace}
    \begin{gathered}
%      I\colon \WonepD{p} \to \RR, \\ 
    I(v) \coloneqq 
    \int_\Omega \phi(\nabla v) 
    - (f,v)_\Omega - (\BF,\nabla v)_\Omega,
    %- \int_\Omega (f v + F\cdot \nabla v) 
%    - \langle g, v \rangle_{\Gamma_N},
  \end{gathered}
  \end{equation}
  where $\phi(\ba)=\phiN(|\ba|)$ for a uniformly convex $N$-function $\phiN$,
  and $f\in L^{\phi^*}(\Omega)$, $\BF\in L^{\phi^*}(\Omega)^d$ are given.
  In this case the dual energy $D\colon F + W^{\phi^*}(\diver;\Omega)\to \Rext$ reads:
  \begin{equation}
D(\br) = - \int_\Omega \phi^*(\br) 
- \characteristic{-f}(\diver(\br-\BF)).
  \end{equation}
Regarding the error measure $\rhotot$, note that the Bregman divergence corresponding to the linear term in $I$ will vanish, and the its conjugate error measure corresponds to the imposition of the divergence constraint and thus vanishes when working with admissible fluxes.
In other words, the total error measure for this problem is simply
\begin{equation}\label{eq:error_tot_phi_laplace}
  \rhotot(v,\br) = 
  \int_\Omega \mathcal{D}_\varphi(\nabla v, \nabla u)
  +
  \int_\Omega \mathcal{D}_{\varphi^*}(\br,\bq)
\end{equation}

  Since $f$ and $\BF$ belong to an infinite-dimensional space, 
  in general one can only produce fluxes satisfying an approximate balance $\diver(\br-\BF_h)=-f_h$,
  where e.g.\ $\BF_h=\Pi^k_h \BF$ and $f_h=\Pi^k_h f$ denote the orthogonal projections of $\BF$ and $f$ onto piecewise polynomials of degree $k\geq 0$.
  To obtain an a posteriori estimate that takes this discrepancy into account,
  we can employ the modified dual energy $\overline{D}\colon \BF_h+W^{\phi^*}(\diver;\Omega)\to \Rext$ given by:
  \begin{equation}
    \overline{D}(\br) = - \int_\Omega \phi^*(\br) 
- \characteristic{-f_h}(\diver(\br-\BF_h)).
  \end{equation}
  This idea was proposed in the linear case in the recent work \cite{GK.2026}.
  Note that, similarly to the non-conforming estimates of the previous section,
  the integrands (and thus the error measures) are unchanged, as we only modified the admissibility criterion.
  With analogous computations as those in Proposition \ref{prop:identity_bregman}, we find that 
  \begin{equation}\label{eq:identity_osc}
  \rhotot(v,\br) = 
\gapext(v,\br) 
+ (f-f_h,u-v)_\Omega 
+ (\BF-\BF_h,\nabla u - \nabla v)_\Omega,
  \end{equation}
  where the modified gap estimator $\gapext$ is defined on admissible fields $v$ and $\br$ as:
  \begin{equation}
\gapext(v,\br) \coloneqq \gap_\varphi(\nabla v,\br) 
= \int_\Omega [\phi(\nabla v)- \br\cdot\nabla v + \varphi^*(\br)].
  \end{equation}
  To handle the leftover terms, the following local Poincaré-like inequality,
  which holds on any $T\in \triang$ and $w\in W^{1,1}(T)$ will be useful \cite{RD.2007}:
  \begin{equation}\label{eq:poincare_orlicz}
    \int_T \widetilde{\phi}(|w-\Pi_h^0w|) \leq 
    C_P \int_T \widetilde{\phi}(h_T|\nabla w|),
  \end{equation}
  where $\widetilde{\phi}$ is an arbitrary uniformly convex $N$-function, and $C_P>0$ is a constant that depends only on the uniform convexity of $\widetilde{\phi}$ and the shape-regularity of $\triang$ (in the quadratic case one can take $C_P=\pi^{-1}$).
Noting that in the second term we can write 
\begin{equation}
  (f-f_h,u-v)_\Omega 
  = \sum_{T\in \triang} (f-f_h,u-v)_T 
  = \sum_{T\in \triang} (f-f_h,u-v- \Pi_h^0(u-v))_T ,
\end{equation}
an application of the Fenchel--Young inequality \eqref{eq:fenchel_young_eps} and the Poincaré-like inequality \eqref{eq:poincare_orlicz} with the $N$-function $\widetilde{\phi}_{|\nabla v|}$ yields with $\varepsilon>0$:
\begin{equation}
  (f-f_h,u-v)_\Omega 
\leq 
\sum_{T\in\triang}\left[\varepsilon \int_T \widehat{\phi}_{|\nabla v|}(|\nabla u - \nabla v|)
+ C_\varepsilon \int_T \widehat{\varphi}^*_{|\nabla v|}(h_T|f-f_h|)
\right]
\end{equation}
Handling the third term in \eqref{eq:identity_osc} in a similar way, and recalling the equivalence from Proposition \ref{prop:young_n_function},
we conclude that the following a posteriori estimate holds for any $v\in W^{1,\phi}_0(\Omega)$ and $\br\in \BF_h+W^{\phi^*}(\diver;\Omega)$ with $\diver(\br-\BF_h)=-f_h$:
\begin{equation}\label{eq:estimate_data_osc}
  \rhotot(v,\br) \lesssim
\gapext(v,\br) 
+
\sum_{T\in \triang} \mathrm{osc}^2_{T;v}(f;\BF),
\end{equation}
where the constant depends only on the uniform convexity constants $c_{\mathrm{uc}},C_{\mathrm{uc}}$ from $\phiN$,
and the shape-regularity of $\triang$ (through $C_P$),
and the \emph{local data oscillation estimator} is defined by 
\begin{equation}\label{eq:data_osc}
  \mathrm{osc}^2_{T;v}(f;\BF) 
  \coloneqq 
  \int_T \widehat{\phi}^*_{|\nabla v|}(h_T|f-f_h|)
+ 
\int_T \widehat{\varphi}^*_{|\nabla v|}(\BF- \BF_h).
\end{equation}

\subsection{The obstacle problem: positivity of the multiplier}\label{sec:obstacle}
We turn now to a non-smooth example where the methodology developed thus far allows us to derive an estimate using a non-feasible field.
Namely, we consider the obstacle problem, where the primal energy is obtained by setting (cf.\ \eqref{eq:indicator}):
\begin{subequations}
\begin{gather}
  \varphi(\cdot,\ba)\coloneqq \tfrac{1}{2}|\ba|^2 
%  \qquad 
  \\
  \psi(\cdot,s) = -f(\cdot)s + \chi_{K(x)}(s),
  =
  \left\{ 
    \begin{array}{cc}
-f(\cdot)s &\text{if }s\geq g(\cdot),\\
      +\infty &\text{otherwise}.
    \end{array}
    \right.,
%  \qquad 
%  \ba\in \RRd,\, s\in \RR.
\end{gather}
\end{subequations}
where $K(x)\coloneqq \{s \in \RR \mid s\geq g(x)\}$, with $g\in H^1(\Omega)$ a given obstacle with $g|_{\partial\Omega}\leq 0$ and $f\in L^2(\Omega)$.
%It is a classical result that the primal problem has a unique solution $u\in H^1_0(\Omega)$,
We shall make some simplifying assumptions just to illustrate the main idea;
a more thorough analysis will be carried out in a subsequent work.
In this respect, assuming that the fluxes $\br$ belong to $\Hdiv$, %(which is the case with usual flux reconstruction procedures),
the dual energy has the representation \eqref{eq:dual_energy2}, with (see e.g.\ \cite{ET.1999,BK.2025})
\begin{subequations}
\begin{gather}
  \varphi^*(\cdot,\bb)=\tfrac{1}{2}|\bb|^2
\\
  \psi^*(\cdot,s) = g(\cdot)(s+f(\cdot)) + \chi_{\geq 0}(-s-f(\cdot))=
  \left\{ 
    \begin{array}{cc}
      g(\cdot)(s+f(\cdot)) &\text{if }-s-f(\cdot)\geq 0,\\
      +\infty &\text{otherwise}.
    \end{array}
    \right..
    \label{eq:obstacle_dual_constraint}
\end{gather}
\end{subequations}
For general $\br$, the $\Hdiv$ regularity assumption is not restrictive, since the reconstructed fluxes corresponding to numerical approximations will usually satisfy this;
on the other hand, for simplicity it is also assumed to hold for the exact flux $\bq =\nabla u$, which amounts to a regularity assumption
(in general $-\diver\bq-f\in H^{-1}(\Omega)$ is a Radon measure).
For a given flux $\br\in \Hdiv$ let us denote $\mu_{\br}\coloneqq -\diver \br-f\in L^2(\Omega)$;
the indicator function in \eqref{eq:obstacle_dual_constraint} means that $\br$ is feasible only if $\mu_{\br}$ is non-negative. 
This problem fits the setting described in Assumption \ref{as:existence}, and thus a generalised Prager--Synge identity holds \cite{RV.2018,BK.2025}:
\begin{equation}\label{eq:ps_obstacle}
  \tfrac{1}{2}\norm{\nabla v - \nabla u}^2_\Omega 
  + (\mu_{\bq},v-u)_\Omega 
  +\tfrac{1}{2}\norm{\br-\bq}^2_\Omega 
  + (\mu_{\br},u-g)_\Omega
  = 
  \tfrac{1}{2}\norm{\nabla v - \br}^2_\Omega 
  + (\mu_{\br},v-g)_\Omega,
\end{equation}
where $v\in K\coloneqq \{v\in \Honezero \mid v\geq g\text{ a.e.\ in }\Omega\}$ and $\br \in K^*\coloneqq \{\bs\in\Hdiv \mid \mu_{\bs}\geq 0 \text{ a.e.\ in }\Omega\}$ are arbitrary.
Generic numerical schemes are not necessarily guaranteed to produce multiplier approximations $\mu_{\br_h}$ that satisfy the positivity constraint,
and are thus inadmissible in the Prager--Synge identity \eqref{eq:ps_obstacle}.
Motivated by the previous sections we then consider the unconstrained dual energy $\overline{D}\colon \Hdiv \to \RR$ given as:
\begin{equation}
  \overline{D}(\br)\coloneqq 
  -\tfrac{1}{2}\norm{\br}^2_\Omega 
  + (g,\mu_{\br})_\Omega.
\end{equation}
Following analogous steps as before,
one obtains the same identity \eqref{eq:ps_obstacle};
note however, that some terms do not represent errors anymore, since $\mu_{\br}$ is not necessarily positive.
Splitting $\mu_{\br}=\mu_{\br_+}-\mu_{\br_{-}}$ into positive and negative parts we can re-write:
\begin{equation}\label{eq:ps_obstacle2}
  \tfrac{1}{2}\norm{\nabla v - \nabla u}^2_\Omega 
  + (\mu_{\bq},v-u)_\Omega 
  +\tfrac{1}{2}\norm{\br-\bq}^2_\Omega 
  + (\mu_{\br_+},u-g)_\Omega
  = 
  \tfrac{1}{2}\norm{\nabla v - \br}^2_\Omega 
  + (\mu_{\br_+},v-g)_\Omega
  + (\mu_{\br_-},u-v)_\Omega
\end{equation}
With a simple application of Young's inequality we can obtain a guaranteed bound with each term representing an error quantity:
\begin{equation}\label{eq:ps_obstacle2_bound}
\norm{\nabla v - \nabla u}^2_\Omega 
  + (\mu_{\bq},v-u)_\Omega 
+\norm{\br-\bq}^2_\Omega 
  + (\mu_{\br_+},u-g)_\Omega
\lesssim
\norm{\nabla v - \br}^2_\Omega 
  + (\mu_{\br_+},v-g)_\Omega
  + \norm{\mu_{\br_-}}^2_{H^{-1}(\Omega)}.
\end{equation}
Hence, for schemes that do not produce positive multipliers,
one only needs the additional ``non-feasibility'' estimator $\norm{\mu_{\br_-}}^2_{H^{-1}(\Omega)}$.
Some numerical methods (e.g.\ the Proximal Galerkin method \cite{KS.2024}) can in fact guarantee positivity of the projection of $\mu_{\br}$ into piecewise polynomials;
using this we can estimate:
\begin{align*}
  \norm{\mu_{\br_-}}_{H^{-1}(\Omega)}
  &= \sup_{v\in \Honezero} \frac{(\mu_{\br_-},v)_\Omega}{\norm{\nabla v}_\Omega}
  = \sup_{v\in \Honezero} \frac{(\mu_{\br_-},v-\Pi_h v)_\Omega}{\norm{\nabla v}_\Omega}
%  \\&
  \leq 
  \sup_{v\in \Honezero} \frac{\sum_{T\in \triang} \norm{\mu_{\br_-}}_T \frac{h_T}{\pi} \norm{\nabla v}_T}{\norm{\nabla v}_\Omega},
\end{align*}
where we used a local Poincaré inequality.
Combining this with the estimate \eqref{eq:ps_obstacle2_bound} above,
we conclude that the (squared) total error can be controled, with an explicit constant, 
using the local estimators:
\begin{equation}
  \begin{gathered}
    \sum_{T\in \triang}\left[ \eta_{\mathrm{cr};T}^2(v,\br)
      + \eta^2_{\mathrm{comp};T}(v,\br)
      + \eta^2_{\mathrm{pos};T}(\br)
    \right]
    \\ 
    \eta_{\mathrm{cr};T}^2(v,\br) \coloneqq \tfrac{1}{2}\norm{\nabla v - \br}^2_\Omega
    \qquad 
    \eta^2_{\mathrm{comp};T}(v,\br) \coloneqq (\mu_{\br_+},v-g)_T
    \qquad 
       \eta^2_{\mathrm{pos};T}(\br) \coloneqq 
       \frac{h_T^2}{\pi^2} \norm{\mu_{\br_-}}^2_T
  \end{gathered}
\end{equation}
which correspond to the flux constitutive relation, the complementarity principle, and the positivity of the multiplier, respectively.

%\newpage
\section{Quasioptimal a priori bounds for uniformly convex energies}\label{sec:quasioptimality}

Let us return to the $\phi$-Laplace problem \eqref{eq:primal_phi_laplace}, now allowing for mixed boundary conditions:
  \begin{equation}
    \begin{gathered}
%      I\colon \WonepD{p} \to \RR, \\ 
    I(v) \coloneqq 
    \int_\Omega \phi(\nabla v) 
    - (f,v)_\Omega - (\BF,\nabla v)_\Omega
    %- \int_\Omega (f v + F\cdot \nabla v) 
    - \langle g, v \rangle_{\Gamma_N},
  \end{gathered}
  \end{equation}
  where $\phi(\ba)=\phiN(|\ba|)$ for a uniformly convex $N$-function $\phiN$,
  and where $f\in L^{\phi^*}(\Omega)$, $\BF\in L^{\phi^*}(\Omega)^d$, $g\in L^{\phi^*}(\Gamma_N)$ are given ($g$ could be also taken in the space dual to the normal traces of $\WdivpN{\phi}$).
  % W^{-\smash{\frac{1}{p'}},p'}_{00}(\Gamma_N)$ are given,
 %  $\phiN$ is a given uniformly convex $N$-function.
  We remark that the linear terms are meant to represent a general functional in $(\WonepD{\phi})^*$,
  and the energy describes the problem with boundary conditions $u|_{\Gamma_D}=0$, $(\bq-\BF)\bn|_{\Gamma_N}=g$.

An appropriate discrete approximation of the energy $I_h\colon \CRDirichlet \to \RR$ (defined on the Crouzeix--Raviart finite element space $\CRDirichlet$) is given by
\begin{equation}
I_h(v_h) \coloneqq 
\int_\Omega \phi(\nabla_h v_h) 
    - (f_h,v_h)_\Omega - (\BF_h,\nabla_h v_h)_\Omega
    - ( g_h, v_h )_{\Gamma_N},
\end{equation}
where $f_h \coloneqq \Pi_h f$, $\BF_h\coloneqq \Pi_h \BF$ and $g_h\coloneqq \pi_h g$ are piecewise constant approximations of the load terms.
We denote the exact solution of the discrete primal minimisation problem as $u_h$;
i.e.\ $u_h$ is the unique function that satisfies $I_h(u_h)=\min_{v_h\in \CRDirichlet}I_h(v_h)$.
Here the associated dual energy (defined on the Raviart--Thomas space $\RTNeumann$) $D_h\colon \BF_h + \RTNeumann \to \Rextm$ takes the form:
\begin{equation}
D_h(\br_h) \coloneqq 
-\int_\Omega \varphi^*(\Pi_h\br_h)
- \characteristic{-f_h}(\diver(\br_h - \BF_h))
- \chi^{\Gamma_N}_{\{g_h\}}((\br_h - \BF_h)\bn)
\end{equation}
And we denote the discrete dual solution by $\bq_h \in K_h \coloneqq \mathrm{dom}(D_h)$.
In this setting we know furthermore that discrete strong duality holds: $I_h(u_h)=D_h(\bq_h)$,
and the optimality relation $\Pi_h \bq_h = \mathcal{A}(\nabla_h u_h)$ is satisfied a.e.\ in $\Omega$
\cite[Prop.~3.1]{Bar.2021}.
Moreover, given the primal solution $u_h\in \CRDirichlet$, the dual solution is available through the generalised Marini formula (for an inverse formula delivering $u_h$ if $\bq_h$ is known, see e.g.\ \cite[Prop.~3.7]{BK.2024}):
\begin{equation}\label{eq:marini_phi_laplace}
  \bq_h = \mathcal{A}(\nabla_h u_h) - \tfrac{f_h}{d}(\mathrm{id}_{\RR^d}- \Pi_h \mathrm{id}_{\RR^d}),
\end{equation}
where $\mathrm{id}_{\RRd}$ is simply the identity map on $\RRd$.

%Now, regarding the error measure $\rhotot$ associated to the continuous problem, note that the Bregman divergence corresponding to the linear term in $I$ will vanish, and the its conjugate error measure corresponds to the imposition of constraints and thus vanishes as well when working with admissible fields.
%In other words, the (continuous) error measure for this problem is simply
%\begin{equation}
%  \rhotot(\nabla v,\br) = 
%  \int_\Omega \mathcal{D}_\varphi(\nabla v, \nabla u)
%  +
%  \int_\Omega \mathcal{D}_{\varphi^*}(\br,\bq)
%\end{equation}
Now, thanks to the fact that $\varphi$ is not $x$-dependent,
the exact same error measure \eqref{eq:error_tot_phi_laplace} can be used at the discrete level.
In fact, we can extend the definition for general broken functions $v\in \Wonepbroken{\phi}$ as
\begin{equation}
  \rhotot(\nabla_h v_h,\br) = 
  \int_\Omega \mathcal{D}_\varphi(\nabla_h v_h, \nabla u)
  +
  \int_\Omega \mathcal{D}_{\varphi^*}(\br,\bq),
\end{equation}
and to measure the discrete error (i.e.\ the distance to the discrete solution $(u_h,\bq_h)$) the same notion of distance can be employed:
\begin{equation}
  \rhototh(\nabla_h v_h,\Pi_h\br_h) \coloneqq
  \int_\Omega \mathcal{D}_\varphi(\nabla_h v_h, \nabla_h u_h)
  +
  \int_\Omega \mathcal{D}_{\varphi^*}(\Pi_h\br_h,\Pi_h\bq_h)
\end{equation}

%We have more explicitly:
%\begin{align*}
%  \rhotot(\nabla_h v_h,\br) &= 
%  \int_\Omega \left[\varphi(\nabla_h v_h) - \varphi(\nabla u)
%  - \mathcal{A}(\nabla u)(\nabla_h v - \nabla u)\right]
%  +
%  \int_\Omega \left[\varphi^*(\br) - \varphi(\bq)
%  - \mathcal{B}(\bq)(\nabla_h v - \nabla u)\right]
%\end{align*}

As for the duality gap estimator, we follow an analogous strategy to Section \ref{sec:local_efficiency}.
Note that only the estimator related to $\varphi$ will appear, since the remaining terms in the energy are linear and so $\mathcal{D}_\psi$ vanishes.
Hence, we define the discrete estimator 
for (discretely) admissible fields $v_h\in \CRDirichlet$, $\br_h\in \BF_h + \RTNeumann$  simply as:
\begin{align*}
\gaph(\nabla_h v_h,\br_h) \coloneqq 
%I_h(v_h) - D_h(\br_h)
\eta^2_{\varphi}(\nabla_h v_h,\Pi_h\br_h )
=
\int_\Omega 
[\varphi(\nabla_h v_h) - 
\Pi_h\br_h\cdot \nabla_h v_h + \varphi^*(\Pi_h\br_h)].
\end{align*}
We remark that the integrand in the discrete estimator is pointwise non-negative, and vanishes if and only if the optimality condition $\Pi_h \br_h = \mathcal{A}(\nabla_h v_h)$ is satisfied.
This makes it potentially useful in the development of inexact solvers, since it characterises the distance to the discrete solution in terms of computable quantities; 
see \cite{DS.2025} for an application of this estimator as an upper bound in the context of the $p$-Laplace problem.

\begin{theorem}[A priori error identity and quasioptimality]\label{thm:quasioptimal}
For any $v_h\in \CRDirichlet$ and $\br_h\in K_h$, the following discrete error identity holds:
\begin{equation}\label{eq:discrete_prager_synge}
\rhototh(\nabla_h v_h,\Pi_h\br_h)
= 
\gaph(\nabla_h v_h ,\Pi_h \br_h).
\end{equation}
Consequently,
we obtain quasioptimality:
\begin{equation}\label{eq:quasioptimality_phi_laplace}
\rhotot(\nabla_h u_h,\Pi_h \bq_h)
\sim
\inf_{v_h \in \CRDirichlet}
\int_\Omega \mathcal{D}_{\phi}(\nabla_h v_h, \nabla u)
  +
  \inf_{\substack{\br_h \in K_h}}
  \int_\Omega   \mathcal{D}_{\phi^*}(\Pi_h\br_h, \bq),
\end{equation}
where the constants depend only on the uniform convexity constants $c_{\mathrm{uc}},C_{\mathrm{uc}}$ of $\varphi$.
\end{theorem}
\begin{proof}
  The proof of the discrete Prager--Synge identity \eqref{eq:discrete_prager_synge} is a straightforward consequence of the discrete strong duality relation:
\begin{equation}
\int_\Omega \varphi(\nabla_h u_h) 
    - (f_h,u_h)_\Omega - (\BF_h,\nabla_h u_h)_\Omega
    - ( g_h, u_h )_{\Gamma_N}
    = 
    - \int_\Omega \varphi^*(\Pi_h \bq_h),
\end{equation}
and the optimality condition $\Pi_h \bq_h = \mathcal{A}(\nabla_h u_h)$.% \Leftrightarrow \nabla_h u_h = \mathcal{B}(\Pi_h \bq_h)$.
%\commentalexei{For general $\br_h \in \Wdivbroken{p'}$ (drop the indicator functions in $D_h$), and $v_h\in \Wonepbroken{p}$ we can prove:
%\begin{equation*}
%  \rhotot(\nabla_h v_h,\Pi_h \br_h)
%=
%\gaph(\nabla_h v_h ,\Pi_h \br_h)
%+ (\br_h-\bq_h,\nabla_h v_h - \nabla_h u_h)_\Omega.
%\end{equation*}
%Could this be useful to get estimates for DG?
%}

To prove the quasioptimality \eqref{eq:quasioptimality_phi_laplace}, first note that the discretisation error can be estimated using the triangle-type inequality \eqref{eq:triangle_unif_convex} as:
\begin{align}
\rhotot(\nabla_h u_h,\Pi_h \bq_h)
&\lesssim 
  \int_\Omega \mathcal{D}_\varphi(\nabla_h v_h, \nabla u)
  +
  \int_\Omega \mathcal{D}_\varphi(\nabla_h v_h, \nabla_h u_h)
  \\
&\quad+
  \int_\Omega \mathcal{D}_{\varphi^*}(\Pi_h\br_h,\bq)
  +
  \int_\Omega \mathcal{D}_{\varphi^*}(\Pi_h\br_h,\Pi_h\bq_h)
\\&= 
  \int_\Omega \mathcal{D}_\varphi(\nabla_h v_h, \nabla u)
  +
  \int_\Omega \mathcal{D}_{\varphi^*}(\Pi_h\br_h,\bq)
  +
\gaph(\nabla_h v_h ,\Pi_h \br_h).
\end{align}
where $v_h\in \CRDirichlet$ and $\br_h\in K_h$ are abritrary;
we employed the discrete error identity \eqref{eq:discrete_prager_synge} in the last step to get rid of the discrete solution $(u_h,\bq_h)$ on the right-hand-side.

Now, we exploit the fact that the error estimator is essentially the same for the continuous and discrete problems;
this will allow us to insert the exact solution on the right-hand-side to obtain best-approximation terms.
The argument to estimate the last term is completely analogous to the proof of Theorem \ref{thm:local_efficiency}:
\begin{align*}
\gaph(\nabla_h v_h ,\Pi_h \br_h)
  &=
  \eta^2_{\phi}(\nabla_h v, \Pi_h\br_h)
  \pm \int_\Omega \phi^*(\bq) 
  \mp \int_\Omega \nabla_h v_h \cdot \bq
\\&
=
\int_\Omega \left[
  \mathcal{D}_{\phi}(\nabla_h v_h, \nabla u)
  - \nabla_h v_h\cdot(\Pi_h\br_h-\bq)
  + \phi^*(\Pi_h\br_h)
  -\phi^*(\bq)
\right]
\\&
=
\int_\Omega \left[
  \mathcal{D}_{\phi}(\nabla_h v_h, \nabla u)
  +
  \mathcal{D}_{\phi^*}(\Pi_h\br_h, \bq)
  - (\nabla_h v_h-\nabla u)\cdot(\Pi_h\br_h-\bq),
\right]
\end{align*}
where we used the continuous optimality condition $\nabla u = \mathcal{A}^{-1}(\bq)$ and
Lemma \ref{lem:bregman}\ref{lem:bregman_sharing_subgradients}.
We can conclude the proof by applying Proposition \ref{prop:young_n_function} and taking the appropriate infima.

\end{proof}

	\printbibliography
	
	%bibtex 
%	\bibliographystyle{alpha.bst}
%  \bibliography{literatur}

\end{document}